\documentclass[11pt]{amsart}

\usepackage{amscd,latexsym,amsthm,amsfonts,amssymb,amsmath,amsxtra,eucal}
\usepackage[all]{xy}
\usepackage[pdftex,bookmarks,colorlinks,pdfmenubar]{hyperref}

\usepackage{verbatim}
\usepackage{mathabx}
\usepackage{mathrsfs}
\usepackage{bbm}

\renewcommand\theequation{\thesection.\arabic{equation}}

\newcommand{\BC}{{\mathbb {C}}}

\newcommand{\BZ}{{\mathbb {Z}}}

\newcommand{\CP}{{\mathcal {P}}}

\newcommand{\Qlb}{\overline{\mathbb{Q}}_{\ell}}

\newcommand{\Aut}{{\mathrm{Aut}}}

\newcommand{\Gal}{{\mathrm{Gal}}}

\newcommand{\rank}{{\mathrm{rank}}}

\newcommand{\Spec}{{\mathrm{Spec}}}

\newcommand{\wt}{\widetilde}

\newtheorem{thm}{Theorem}[section]
\newtheorem{cor}[thm]{Corollary}
\newtheorem{lem}[thm]{Lemma}
\newtheorem{prop}[thm]{Proposition}

\newtheorem {ques/conj}[thm]{Question/Conjecture}

\newtheorem{defn}[thm]{Definition}

{\theoremstyle{definition}
\newtheorem{rmk}[thm]{Remark}}

\newtheorem{exmp}[thm]{Example}

\usepackage[numbers]{natbib}

\newcommand{\select}[1]{{\it{#1}}}

\begin{document}
\renewcommand{\theequation}{\arabic{equation}}
\numberwithin{equation}{section}

\title[almost multiplicity-one]{Almost multiplicity-one property of spherical varieties over finite fields}

\date{\today}

\author[Fulin Chen]{Fulin Chen$^1$}
\address{School of Mathematical Sciences, Xiamen University, Xiamen 361005, Fujian, P.R. China}
  \email{chenf@xmu.edu.cn}
\thanks{1. Partially supported by China NSF grant (No. 12471029) and Xiamen NSF grant (No. 3502Z202473005)}

\author[Fang Shi]{Fang Shi$^2$}

\address{School of Mathematical Sciences, Xiamen University, Xiamen 361005, Fujian, P.R. China}

\email{11935007@zju.edu.cn}

\thanks{2. Partially supported by the Postdoctoral Fellowship Program of CPSF (No. GZC20252027)}

\author[Shaobin Tan]{Shaobin Tan$^3$}
\address{School of Mathematical Sciences, Xiamen University, Xiamen 361005, Fujian, P.R. China}
 \email{tans@xmu.edu.cn}
 \thanks{3. Partially supported by China NSF grant (No. 12131018)}


\maketitle

\begin{abstract}
Let $H$ be a connected algebraic subgroup of a connected reductive group $G$ over a finite field $\mathbb F_q$ such that $G/H$ is a $G$-spherical variety, i.e.,  $G/H$ has an open dense $B$-orbit for each Borel subgroup $B$ of $G$. 
We formulate, for the pair $(G,H)$,  an almost multiplicity-one property. Then we establish a criterion for this  property in terms of the $B$-stabilizers on $G/H$. In particular, we will see that this property is analogous to the strongly tempered condition in characteristic $0$.
\end{abstract}

\section{Introduction}

\subsection{Background}
Let  $\mathfrak {G}$ be a finite group and $\mathfrak{H}$  a subgroup of $\mathfrak {G}$. A fundamental problem in representation theory is to study the decomposition of  an irreducible representation  $\pi$ of $\mathfrak {G}$ into irreducible constituents  when restricted to  $\mathfrak {H}$. It is natural to ask whether such a decomposition is multiplicity free, and how frequently it fails to be multiplicity free as $\pi$ varies. We often say that the pair $(\mathfrak{G},\mathfrak{H})$ has the multiplicity-one property if for every irreducible representations $\pi$ of $\mathfrak{G}$, the trivial representation of $\mathfrak{H}$ occurs in $\pi|_{\mathfrak{H}}$ (the restriction of $\pi$ to $\mathfrak{H}$) with multiplicity at most $1$. 

The recent paper \cite{BZSV} proposes a duality in the relative Langlands program, recovering at a numerical level the relationship between a period on a reductive group $\mathcal G$ and an $L$-function attached to its dual group $\check {\mathcal G}$. An interesting example is the duality between the GGP period \cite{GGP} and the equal-rank $\theta$-correspondence \cite{Ra}. Notably, spherical varieties play a crucial role in \cite{BZSV} and \cite{SV}. It is shown in \cite{WZ} that for many strongly tempered spherical varieties in question, the summation of the multiplicities is always equal to $1$ over every local tempered $L$-packet. This motivates studying the corresponding periods problem in the finite-field setting.

 We  consider some finite-field counterparts of a class of periods introduced in \cite{BZSV}, and study their multiplicity-one property. However, this property may  be ill-defined: for a fixed spherical pair over a finite field, it is common that not all irreducible representations have this property  even when the property is well-established for the corresponding number-field prototype (though it holds for ``almost all'' irreducible representations of the finite Lie group). See \cite{R} and \cite{W} for the GGP problem as an example. This suggests that we should formulate the multiplicity-one property in a weaker sense.

In this paper, we formulate an almost multiplicity-one property, and deduce a criterion for this property in terms of the orbits of a Borel subgroup on the spherical variety and their stabilizers. In particular, it is straightforward to see that the finite-field analog of various strongly tempered spherical varieties considered in \cite{WZ} (including the basic cases of Bessel models in the sense of \cite{GGP}) fulfills this criterion. 
\subsection{Main results and examples}\label{subsec-intro-results}
Fix a finite field $\mathbb F_q$, where $q$ is a power of a prime. Fix an algebraic closure $\mathrm k$ of $\mathbb F_q$.
Fix a connected reductive group $G_0$ (over $\mathbb F_q$), and let $H_0$ be a connected algebraic subgroup of $G_0$. Let $G$ (resp., $H$) be the pullback of $G_0$ (resp., $H_0$) to $\mathrm k$. Suppose further that $\mathcal X:=G/H$ is a spherical variety with respect to the left-$G$-action: that is, there exists an open dense $B$-orbit on $G/H$ for each Borel subgroup $B$ of $G$. (We also say that $H$ is a spherical subgroup of $G$.) Note that both $G$ and $H$ are equipped with a geometric Frobenius endomorphism $F$. For a positive integer $n$, let $F^n$ denote the $n$-th power of $F$, and let $G^{F^n}$ (resp., $H^{F^n}$) denote the fixed-point set of the action induced by $F^n$ on $G(\mathrm k)$ (resp., $H(\mathrm k)$). 

For a finite (abstract) group $K$, define the inner product $\langle \cdot,\cdot\rangle_K$ by
$$\langle\pi_1,\pi_2\rangle_K=\frac{1}{|K|}\sum\limits_{k\in K}\pi_1(k)\overline {\pi_2(k)},$$ where $\pi_1$ and $\pi_2$ are class functions on $K$. 
We often do not distinguish a finite-dimensional representation of $K$ from its character. Let $1_K$ denote the trivial representation of $K$. For a finite set $S$, let $\# S$ denote its cardinality.

\begin{defn}\label{de-mainratio}\label{def-almostmulone} For a positive integer $n$,
 let   $\mathcal {E} (G,n)$ be the set of (equivalence classes of) irreducible representations of $G^{F^n}$, and  let $\mathcal E^\diamondsuit(G,H,n)$ be the set of (equivalence classes of) irreducible representations $\pi$ of $G^{F^n}$ satisfying 
    $$
    \langle\pi,1_{H^{F^n}}\rangle_{H^{F^n}}=1.$$
    We say that the almost multiplicity-one property holds for $(G,H)$ (or holds for $H$, for simplicity)  if 
    $$
    \lim_{m\to \infty} \frac{\#\mathcal E^{\diamondsuit}(G,H,m)}{\#\mathcal E(G,m)}=1.
    $$
\end{defn}

For a Borel pair $(B,T)$ of $G$, we set $\mathrm d_{T}^B:B\to T$ to be the quotient map that realizes $T$ as the reductive quotient of $B$ and provides a section for the natural inclusion $T\hookrightarrow B$. For an algebraic subgroup $C$ of $B$, let $\mathrm d_T^B(C)$ denote the image of $C$ under the map $\mathrm d_T^B$; here, we view $\mathrm d_T^B(C)$ as an algebraic subgroup of $T$. For $w\in B(\mathrm k)\backslash G(\mathrm k)/H(\mathrm k)$, we sometimes write $B\cap w H w^{-1}$ to denote the group $B\cap \wt w H \wt w^{-1}$, where $\wt w\in G(\mathrm k)$ is a representative of $w$. In such cases, the relevant results are independent of the choice of $\wt w$.

The following is a corollary of Theorem \ref{thm-main}. 

\begin{thm}\label{thm-intro}
    The following statements are equivalent. 
    \begin{itemize}
        \item[(i)] The almost multiplicity-one property holds for $H$.
        \item[(ii)] Fix any Borel pair $(B,T)$ of $G$, there exists a unique $w\in B(\mathrm k)\backslash G(\mathrm k)/H(\mathrm k)$ such that  $$\mathrm d_T^B(B\cap w Hw^{-1})$$ is a finite algebraic group; moreover, for this specific $w$, the algebraic group $B\cap w Hw^{-1}$ is  trivial.
        (In this theorem, we endow $B\cap w Hw^{-1}$ with the reduced scheme structure for every $w\in B(\mathrm k)\backslash G(\mathrm k)/H(\mathrm k)$.)
    \end{itemize}
\end{thm}

Note that condition (ii) in Theorem \ref{thm-intro} implies the connectedness of the generic $B$-stabilizer on $G/H$.

\begin{rmk}
    In Section \ref{sec-0-counterpart}, we will see that condition (ii) in Theorem \ref{thm-intro} serves as an analog of the characteristic $0$ condition requiring that the dual group $\check{G}_{\mathcal X}$ of the spherical variety $\mathcal X$ equals the dual group $\check{G}$ of $G$. 
\end{rmk}

\begin{exmp}Fix a positive integer $n$. It is well known that $\mathrm {Sp}_{2n}/\mathrm{GL}_n$ is a spherical variety of $\mathrm {Sp}_{2n}$.
In this example, let $(G,H)=(\mathrm {Sp}_{2n},\mathrm{GL}_n)$.  Fix a Borel subgroup $B$ of $G$.
    
    We first consider the case where the base field has characteristic $p\neq 2$.  
    The generic $B$-stabilizer is disconnected. Consequently, by the above theorem, the almost multiplicity-one property does not hold for $(G,H)$.
    
    We next consider the case where the base field has characteristic $p=2$. 
  A tedious exercise in linear algebra shows that  condition (ii) in Theorem \ref{thm-intro} holds true for $(G,H)$. Consequently, the almost multiplicity-one property holds for $(G,H)$. 
    
\end{exmp}

\begin{rmk}\label{rm-intro}
    We note that the condition  (ii) in Theorem \ref{thm-intro} is  geometrical. In particular, we have:
    \begin{itemize}
        \item[~] 
        For two pairs $(H_0^1, G_0^1)$ and $(H_0^2 , G_0^2)$ over $\mathbb{F}_q$ with $H_0^1$ (resp., $H_0^2$) spherical in $G_0^1$ (resp., $G_0^2$), if the pairs $(H^1 , G^1)$ and $(H^2,G^2)$ are isomorphic over $\mathrm{k}$, then the almost multiplicity-one property holds for $(G^1,H^1)$ if and only if it holds for $(G^2,H^2)$.
    \end{itemize}
    
\end{rmk}

\begin{exmp}\label{example-U-VS-GL}
    In this example we consider the pairs $(\mathrm {GL}_n\times \mathrm {GL}_{n-1} ,\mathrm {GL}_{n-1})$ and $(\mathrm U_n\times \mathrm U_{n-1},\mathrm U_{n-1})$. The corresponding periods are finite-field analogs of the basic cases of Bessel models in the sense of \cite{GGP}. In view of Remark \ref{rm-intro}, the almost multiplicity-one property holds for $(\mathrm {GL}_n\times \mathrm {GL}_{n-1} ,\mathrm {GL}_{n-1})$ if and only if it holds for $(\mathrm U_n\times \mathrm U_{n-1},\mathrm U_{n-1})$. (This property is known to hold---a fact well established among experts---and the condition (ii) in Theorem \ref{thm-intro} can be verified directly for these cases.)
\end{exmp}

\subsection{Structure of this paper}

Sections 2--5 are dedicated to proving Theorem \ref{thm-main}. In Section \ref{sec-regular-dl}, we review \cite{DL} and  show  that almost all irreducible representations (identified with their characters) are, up to a sign, regular Deligne-Lusztig characters.

In Section \ref{sec-formulation-efficient}, we formulate the almost multiplicity-one property and deduce a necessary condition, which we rephrase as $i)_B+ii)_B$ in Remark \ref{rm-changeBcondition}; note that $i)_B+ii)_B$ is a priori weaker than the condition (ii) introduced in Theorem \ref{thm-intro}. The methods adopted in Section \ref{sec-formulation-efficient} are elementary: the classical Mackey formula plays the most important role. 

Having established a necessary condition via Mackey theory, we adopt a cohomological perspective in Section \ref{sec-cohomological-interpret}. 
Using Deligne's weight theory,  we show in Theorem \ref{thm-pertopcohomo} that the multiplicity in question for a Deligne-Lusztig character is indeed the trace of the Frobenius operator on an associated top cohomology space. In Section \ref{sec-thegeometry}, we study the geometry of specific objects arising from Section \ref{sec-cohomological-interpret} and complete the proof of the main theorem: Theorem \ref{thm-almostmulone} establishes that $i)_B+ii)_B$ implies condition (i) of Theorem \ref{thm-intro}, while Theorem \ref{thm-refine-main} strengthens $i)_B + ii)_B$ to condition (ii) in Theorem \ref{thm-intro}. The proof of Theorem \ref{thm-refine-main} rests on  properties of the Steinberg representation and a point-counting procedure.  The conclusion is stated as Theorem \ref{thm-main}.

Finally,  Section \ref{sec-0-counterpart} examines the characteristic $0$ counterpart---the strongly tempered hyperspherical varieties in the sense of \cite{BZSV,WZ}. Remark \ref{rm-strtempered} and the subsequent discussion show that the condition $i)_B+ii)_B$ precisely corresponds to the requirement that the dual group $\check G_\mathcal X$ of $\mathcal X$ equals the dual group $\check G$ of $G$.
The analysis relies on \cite[Theorem 2.3]{Kno2} and \cite[Theorem 6.2]{Kno}.

\subsection{Convention}\label{s-convention}
For a scheme $X_0$ over $\mathbb F_q$, we will frequently denote its pullback to $\mathrm k$ by $X$. Then we have the geometric Frobenius endomorphism $F:X\to X$. Let $X^F$ denote the fixed-point set of $F$ on $X(\mathrm k)$.


In the situation of the previous paragraph, we will frequently consider the scheme $X_0\times_{\Spec (\mathbb F_q)}\Spec (\mathbb{F}_{q^n})$ for positive integers $n$. We view $\mathbb F_{q^n}$ as a subfield of $\mathrm k$. We have a natural identification $$\left(X_0\times_{\Spec (\mathbb F_q)}\Spec (\mathbb{F}_{q^n})\right)\times_{\Spec (\mathbb{F}_{q^n})}\Spec(\mathrm k)\cong X.$$ Then the geometric Frobenius endomorphism with respect to this pullback is identified with $F^n$, the $n$-th power of the endomorphism $F$ introduced in the above paragraph. Let $X^{F^n}$ denote the fixed-point set of $F^n$ on $X(\mathrm k)$.

We adopt the notation introduced in the first two paragraphs of Subsection \ref{subsec-intro-results} throughout this paper.

For a Weil sheaf $(\mathscr F, a:F^* \mathscr F\xrightarrow{\sim} \mathscr F)$ on $X$ and an integer $i$, we set 
$$
\mathrm {Tr}(F^*,H_c^{i}(X,\mathscr F)) 
$$
to be the trace of the endomorphism on $H_c^{i}(X,\mathscr F)$ induced by applying $H_c^{i}(X,\_)$ to the composition $$
\mathscr F \to F_*F^* \mathscr F\to F_* \mathscr F,
$$ 
where the first arrow is induced by the adjunction, and the second is induced by the Weil structure $a:F^* \mathscr F \xrightarrow{\sim} \mathscr F$. Also, we set 
$$
\mathrm {Tr}(F^*,H_c^{\bullet}(X,\mathscr F)) =\sum_{i\in \mathbb Z}(-1)^i\mathrm {Tr}(F^*,H_c^{i}(X,\mathscr F)),
$$
where on the right-hand side, there must be only finitely many $i$ such that $H_c^{i}(X,\mathscr F)$ is non-vanishing.

For a maximal torus $T$ of $G$, let $W(T)$ denote the Weyl group of $T$. By definition, we have $W(T):=N(T,T)/T$, where $N(T,T):=\{g\in G:gTg^{-1}=T\}$. We view $W(T)$ as a finite abstract group. When $T$ is assumed to be $F$-stable,  we view $W(T)$ as a finite abstract group endowed with an action of the geometric Frobenius in a natural way, and let $W(T)^F$ denote the fixed-point group.

We sometimes denote a Borel pair of $G$ by $(\mathrm B,\mathrm T)$ in roman type to emphasize that both $\mathrm T$ and $\mathrm B$ are $F$-stable.

We sometimes say that $\mathrm W$ is the Weyl group of $G$. This means that we can choose any $F$-stable Borel pair $(\mathrm B,\mathrm T)$ and define $\mathrm W=W(\mathrm T)$. (In such cases, the corresponding results are independent of the choice of $(\mathrm B , \mathrm T)$.)

The identity element of  an algebraic group is frequently denoted by $e$.


For a reductive group $K_0$ over $\mathbb F_q$, we define $\sigma(K)$ to be the $\mathbb F_q$-rank of $K$, where $K$ denotes the pullback of $K_0$ to $\mathrm k$.

We introduce the following convention for simplicity. Fix for the moment a family $\mathbf{A}_1,\ldots,\mathbf{A}_n$ of statements, where $n$ is a positive integer. Let $\mathbf{A}_1+\ldots+\mathbf{A}_n$ denote the statement such that $\mathbf{A}_1+\ldots+\mathbf{A}_n$ holds true if and only if $\mathbf{A}_i$ holds true for $1\leq i\leq n$.

\section{Preliminaries on Deligne-Lusztig Characters}\label{sec-regular-dl}
 In this section, we show that ``almost all" irreducible representations are regular Deligne-Lusztig characters.


\subsection{Regular Deligne-Lusztig Characters}
In this subsection, let $T$ be an $F$-stable maximal torus of $G$ and $\chi:T^F\to \Qlb^\times$ be a character. The virtual character $R_{T,\chi}$ (which is denoted by $R_{T}^\chi$ in \cite{DL}) of $G^F$ is defined in \cite{DL}.

\begin{defn}\label{def-regularch}
    We say that the character $\chi$ is regular if $\chi$ is not kept fixed by any nontrivial element in $W(T)^F$. 
\end{defn}

  Notice that the notion of ``regular" introduced in Definition \ref{def-regularch} is equivalent to the notion of ``in general position" introduced in \cite[Definition 5.15]{DL}.

\begin{defn}\label{def-regulardl}
   We say that the virtual character $R_{T,\chi}$ is a regular Deligne-Lusztig character if $\chi$ is regular.
 \end{defn}

The following follows from \cite[Proposition 7.4]{DL}.
\begin{prop}\label{regulardl-irr}
    Let the notation be as in Definition \ref{def-regulardl}. Suppose that $R_{T,\chi}$ is a regular Deligne-Lusztig character. Then $(-1)^{\sigma(G)+\sigma(T)}R_{T,\chi}$ is the character of an irreducible representation of $G^F$.
\end{prop}

For future use, we record the following results.

\begin{lem}\label{l-paraind}
    Suppose that $T$ is contained in an $F$-stable Borel subgroup $B$ of $G$. Then $R_{T,\chi}$ is the character of the parabolic induction $\mathrm{Ind}^{G^F}_{B^F}\chi$, where we view $\chi$ as a representation of $B^F$ via the natural quotient map $B^F\to T^F$.
\end{lem}

\begin{proof}This is well known. See \cite[1.20 and 1.17]{DL} and \cite[Corollary 4.3]{DL}.
\end{proof}

 The following follows from \cite[Corollary 7.7]{DL}.
\begin{lem}\label{dl-cor7.7}
    For any irreducible representation $\rho$ of $G^F$, there exist an $F$-stable torus $S$ and a character $\theta$ of $S^F$ such that $\langle \rho, R_{S,\theta}\rangle_{G^F}\neq 0$.
\end{lem}

As in  \cite[Section 6]{DL}, for another $F$-stable maximal torus $T'$ of $G$, we set $$N(T,T'):=\{g\in G:Tg=gT'\}$$
and
$$W(T,T')^F=T^F\backslash N(T,T')^F.$$
For the intertwining number, we have the following theorem, see \cite[Theorem 6.8]{DL}.
\begin{thm}\label{dl-th6.8}
    Let  $T'$ be another $F$-stable maximal torus of $G$, and let $\chi'$ be a character of ${T'}^F$. Then 
    $$
    \langle R_{T}^\chi,R_{T'}^{\chi'}\rangle_{G^F}=\# \{
    w\in W(T,T')^F:~  (\mathrm{ad}~w)(\chi')=\chi
    \}.
    $$

\end{thm}

\subsection{Counting irreducible representations}

In this subsection, let $n$ be a positive integer.
The pullback of $G_0\times_{\Spec (\mathbb F_q)}\Spec(\mathbb F_{q^n})$ to $\mathrm k$ is naturally identified with $G$, and the geometric Frobenius with respect to this pullback is naturally identified with $F^n$. 

\begin{rmk}
  Fix an $F$-stable maximal torus $T$ of $G$ and a character $\chi:T^F\to \Qlb^\times$.  When we say that $R_{T,\chi}$ is a Deligne-Lusztig character of $G^F$, we implicitly specify the reductive group $G_0$ over the finite field $\mathbb F_q$ and the field extension $\mathbb F_q\subset \mathrm k$. Hence we also say that $R_{T,\chi}$ is a Deligne-Lusztig character with respect to the pair $(\mathbb F_q\subset \mathrm k,G_0)$.
\end{rmk}

    We say that a virtual character $\rho$ of $G^{F^n}$ is a Deligne-Lusztig character if it is a Deligne-Lusztig character with respect to the pair $\left(\mathbb F_{q^n}\subset \mathrm k,G_0\times_{\Spec(\mathbb F_q)}\Spec(\mathbb F_{q^n}) \right)$. 

  For an $F^n$-stable maximal torus $T$ of $G$  and a character $\chi:T^{F^n}\to \Qlb^\times$, we denote the corresponding Deligne-Lusztig character of $G^{F^n}$ by $R_{T,\chi}^{(n)}$. Parallel to Definition \ref{def-regularch}, we say that $\chi$ is regular if it is not kept fixed by any nontrivial elements in $W(T)^{F^n}$. Also, we say
  that $R_{T,\chi}^{(n)}$ is a regular Deligne-Lusztig character if $\chi$ is regular.

 Note that If $n=1$, we have $R_{T,\chi}=R_{T,\chi}^{(1)}$.

\begin{defn}\label{de-regratio}
    We set $\mathrm P_{DL}(n)$ to be the ratio
    $$
   \frac{ \#\{\text{regular Deligne-Lusztig characters of $G^{F^n}$}\}}{\#\{\text{isomorphism classes of irreducible representations of $G^{F^n}$}\}}.
   $$
\end{defn}

\begin{defn}\label{def-pdlcirc} Fix an $F$-stable Borel pair $(\mathrm B,\mathrm T)$.
     For a positive integer $n$, we set $\mathrm P^\circ_{DL}(n)$ to be the ratio
    
    $$
   \frac{ \#\{\mathrm {Ind}_{\mathrm B^{F^n}}^{G^{F^n}}\chi: \chi \text{ is a regular character of $\mathrm{T}^{F^n}$} \}}{\#\{\text{isomorphism classes of irreducible representations of $G^{F^n}$}\}}.
   $$
  Here we view a character $\chi$ of $\mathrm{T}^{F^n}$ as a character of $\mathrm{B}^{F^n}$ via the natural quotient $\mathrm{B}^{F^n}\to \mathrm{T}^{F^n}$.
\end{defn}

\begin{rmk}\label{rm-numberNTX}
Fix a torus $T$ over $\mathrm k$ and its closed subscheme $X$. In this remark we recall the number $\mathbf N(T,X)$ defined in \cite[Lemma 8.2]{L5}.
    Let $X_1,X_2,\ldots,X_m$ be the irreducible components of $X$. For each nonempty $I\subset \{1,2,\ldots,m\}$, set $X_I=\bigcap \limits_{i\in I}X_i$, with $X^{(1)}_I,X^{(2)}_I,\ldots,X^{(f(I))}_I$ the connected components of $X_I$. Then we set $\mathbf N(T,X):=\sum\limits_{I} f(I)$, where the summation index $I$ ranges over the set of non-empty sets of $\{1,2,\ldots,m\}$.
\end{rmk}

The following proposition is well known. We include a proof for completeness.
\begin{prop}\label{pro-dl-vs-irrerep}
    We have the following identity
    $$
    \lim\limits_{n\to \infty}\mathrm P_{DL}(n)=1.
    $$
\end{prop}

\begin{proof}

During this proof, let $n$  be a variable positive integer. But in each of the following paragraphs, the integer $n$ is assumed to be fixed.

   Fix an $F^n$-stable maximal torus $S$ of $G$ in this paragraph.  We say that a regular Deligne-Lusztig character $\rho$ of $G^{F^n}$ corresponds to $S$  if $\langle \rho,R_{S,\chi}^{(n)}\rangle_{G^{F^n}}\neq 0$ for some character $\chi$ of $S^{F^n}$. We say that an irreducible representation $\pi$ of $G^{F^n}$ is relevant to $S$ if $\langle \pi,R_{S,\eta}^{(n)}\rangle_{G^{F^n}}\neq 0$ for some character $\eta$ of $S^{F^n}$.
   
   By Theorem \ref{dl-th6.8}, any regular Deligne-Lusztig character of $G^{F^n}$ corresponds to a unique $F^n$-stable maximal torus of $G$ up to $G^{F^n}$-conjugation.

  By Lemma \ref{dl-cor7.7},  it suffices to estimate the ratio 
    $$F_{n,T}:=
   \frac{ \#\{\text{regular Deligne-Lusztig characters of $G^{F^n}$ that correspond to $T$}\}}{\#\{\text{isomorphism classes of irreducible representations of $G^{F^n}$ that are relavant to $T$}\}}
   $$
   for each $F^n$-stable maximal torus $T$ of $G$,
since we clearly have
\begin{equation}\label{ineq-pdl}
    \mathrm P_{DL}(n)\geq \min\limits_{(T),n}\{F_{n,T}\},
\end{equation}
where the subscript $(T),n$ means that $T$ ranges over the $G^{F^n}$-conjugacy classes of $F^n$-stable maximal tori of $G$.

   In the remainder of this proof, for any $F^n$-stable maximal torus $T$ of $G$, we view a character of $T^{F^n}$ as an element in $(T^*)^{F^n}$, where $T^*$ is the dual torus of $T$. (See \cite[(5.21.5)]{DL}.)
Let $^wT^*$ denote the fixed-point subscheme of $T^*$ under $w\in W(T)$, endowed with the reduced scheme structure.

In the remainder of this proof, let $(\mathrm B,\mathrm T)$ be an $F$-stable Borel pair  of $G$. Let $X_{\mathrm T^*}$ be the minimal reduced closed subscheme of $\mathrm T^*$ satisfying:
\begin{itemize}
    \item $X_{\mathrm T^*}$ contains $^v \mathrm T^*$ for every \textbf{non-trivial} $v\in W(\mathrm T)$;
    \item $X_{\mathrm T^*}$ is stable under the geometric Frobenius $F$.
\end{itemize}
Then every irreducible component of $X_{\mathrm T^*}$ is a  translate of a proper subtorus of $\mathrm T^*$.

For any $F^n$-stable maximal torus $T$ of $G$, we can choose $g_T\in G(\mathrm k)$ such that $\mathrm Tg_T=g_T T$. Set $c_T:T\to\mathrm T$ to be the isomorphism defined by $t\mapsto  g_Ttg^{-1}_T$. Let $c^\mathrm t_T:\mathrm T^*\to T^*$ be the transpose map of $c_T$. In view of the explicit relation between  Frobenius structures on $T$ and $\mathrm T$ (see \cite[Corollary 1.14]{DL} and what follows), we see that $c^\mathrm t_T(X_{\mathrm T^*})$ contains $^w T^*$ for every nontrivial $w\in W(T)$, and that $c^\mathrm t_T(X_{\mathrm T^*})$ is $F^n$-stable.
Consequently, by \cite[Lemma 8.2]{L5}, we have\begin{equation}\label{ineq-regularchara}
    \begin{aligned}
    &\#\{\chi\in (T^*)^{F^n}:\text{$\chi$ is regular as a character of $T^{F^n}$}\}\\
\geq &\#\left((T^*)^{F^n}-(c^\mathrm {t}_T(X_{\mathrm T^*}))^{F^n}\right)\\ \geq &(1-\mathbf N(\mathrm T^*,X_{\mathrm T^*})\frac{(q^n+1)^{\dim(T)-1}}{(q^n-1)^{\dim (T)}})\#((T^*)^{F^n}). 
\end{aligned}
\end{equation}
(See Remark \ref{rm-numberNTX} for the definition of $\mathbf N(\_,\_)$, and note that $\mathbf N(\mathrm T^*,X_{\mathrm T^*})=\mathbf N(T^*,c^\mathrm t_{T}(X_{\mathrm T^*}))$ by the definition.)
By  Theorem \ref{dl-th6.8}, we see that
$$
F_{n,T}\geq \frac{\#\left((T^*)^{F^n}-(c^\mathrm {t}_T(X_{\mathrm T^*}))^{F^n}\right)}{\#(T^*)^{F^n}+\#\left(W(T)\right)\#\left((c^\mathrm {t}_T(X_{\mathrm T^*}))^{F^n})\right)}.
$$
By the inequalities (\ref{ineq-pdl}) and (\ref{ineq-regularchara}),  we see that $$ 1\geq \limsup_{m\to \infty} \mathrm P_{DL}(m)\geq\liminf_{m\to \infty} \mathrm P_{DL}(m)\geq
\liminf_{m\to\infty} \min_{(T),m}\{ F_{m,T}\}=1,
$$
as desired.
\end{proof}

\begin{cor}\label{cor-regularprincipal}
    We have the following inequality
    $$
    \liminf\limits_{m\to \infty} \mathrm P^\circ_{DL}(m)\geq \frac{1}{\#\mathrm W},
    $$
    where $\mathrm W$ denotes the Weyl group of $G$.
\end{cor}

\begin{proof} In this proof, we retain the convention adopted in the proof of Proposition \ref{pro-dl-vs-irrerep}.

    We fix an $F$-stable Borel pair $(\mathrm B,\mathrm T)$ of $G$, as in  Definition \ref{def-pdlcirc}. 
    Let $I_{n,\mathrm T}$ denote the ratio
    $$
 \frac{\#\{\text{regular Deligne-Lusztig characters of $G^{F^n}$ that correspond to $\mathrm T$}\}}{\#\{\text{regular Deligne-Lusztig characters of $G^{F^n}$ }\}}.
    $$
By Proposition \ref{pro-dl-vs-irrerep}, it suffices to show that
$$
\liminf_{m\to \infty} I_{m,\mathrm T}\geq \frac{1}{\#\mathrm W}.
$$
The corollary follows from the classification of $F^n$-stable maximal tori \cite[Corollary 1.14]{DL} and the estimation (\ref{ineq-regularchara}) of the number of regular characters. 

\end{proof}

\section{Formulations of the Almost Multiplicity-one Property }\label{sec-formulation-efficient}
Recall the almost multiplicity-one property defined in Definition \ref{def-almostmulone}. In this section, we deduce a necessary condition for this property. 

In this section, let $T$ be an $F$-stable maximal torus of $G$ contained in an $F$-stable Borel subgroup $B$ of $G$, and let $\mathrm d:B\to T$ be the natural quotient map.
\subsection{Principal-series representations}

In this subsection, let $n$ be a postive integer.

We recall the following definition for later use.

\begin{defn}
We define  $$\mathcal P(G,n):=\{\mathrm{Ind}_{B^{F^n}}^{G^{F^n}}(\chi): \text{$\chi$ is a character of $T^{F^n}$} \}$$  and refer to elements in $\mathcal P(G,n)$ as {\bf{principal-series}} representations of $G^{F^n}$.
If in addition the character $\chi:T^{F^n}\to \Qlb^\times$ is regular in the sense of Definition \ref{def-regularch},  we call  $\mathrm {Ind}_{B^{F^n}}^{G^{F^n}}(\chi)$  a regular principal-series representation. 
(In the definition of $\mathcal P(G,n)$, we view characters of $T^{F^n}$ as representations of $B^{F^n}$ via the natural quotient $B^{F^n}\to T^{F^n}$.)
\end{defn}

The following is a variant of the almost multiplicity-one property regarding principal series representations.

\begin{defn}\label{def-almostmulone-principal}Let $\mathcal P^\diamondsuit(G,H,n)$ denote the set of (equivalence classes of) regular principal-series representations $\pi$ of $G^{F^n}$ satisfying 
    $$
    \langle\pi,1_{H^{F^n}}\rangle_{H^{F^n}}=1.$$
    We say that the almost multiplicity-one property holds for $H$ with respect to principal-series representations if 
    $$
    \lim_{n\to \infty} \frac{\#\mathcal P^{\diamondsuit}(G,H,n)}{\#\mathcal P(G,n)}=1.
    $$
\end{defn}

We have the following remark concerning double coset spaces.

\begin{rmk}\label{rm-rationorbit-vs-geoorbit}
    There is a natural map $\Phi_q$ from $B^{F}\backslash G^F/H^F $ to $\left(B(\mathrm k)\backslash G(\mathrm k)/H(\mathrm k)\right)^F$, where $\left(B(\mathrm k)\backslash G(\mathrm k)/H(\mathrm k)\right)^F$ denotes the fixed-point set of $B(\mathrm k)\backslash G(\mathrm k)/H(\mathrm k)$ under the natural action of the Frobenius $F$.  The map $\Phi_q$ is surjective by Lang's theorem.
    Fix $w\in B^{F}\backslash G^F/H^F$ and let $\bar w=\Phi_q(w)$.
     The fibre $\Phi_q^{-1}(\{\bar w\})$ is in bijection with the set of $F$-conjugacy classes in $\pi_0(B\cap wHw^{-1})$. (See \cite[Proposition 1.16]{DL}.) Here, let $\pi_0(\_)$ denote the component group.     
\end{rmk}

For future use, let $\Phi_{q^n}$ denote the natural map  from $B^{F^n}\backslash G^{F^n}/H^{F^n} $ to $\left(B(\mathrm k)\backslash G(\mathrm k)/H(\mathrm k)\right)^{F^n}$. Note that the similar statements to those in Remark \ref{rm-rationorbit-vs-geoorbit} hold for the map $\Phi_{q^n}$ when the geometric Frobenius in question is $F^n$.

\subsection{A necessary condition} In this subsection, we formulate a necessary condition for the almost multiplicity-one property.

We need the following lemma.
\begin{lem}\label{lem-imageFpoints}
    Let $f_0:X_0\to Y_0$ be a morphism of algebraic groups over $\mathbb F_q$, with the kernel $Z_0$. Let $f:X\to Y$ denote the pullback of $f_0$ to $\mathrm k$. Let $Z$ be the pullback of $Z_0$ to $\mathrm k$. Let $\mathrm {I}_f$ be the image of $X^F$ under $f$, where we view $\mathrm {I}_f$ as a subgroup of $Y^F$. Suppose that $f$ is surjective. Then we have $\#\mathrm{I}_f\geq \#Y^F/K_f$, where $K_f$ is the cardinality of the  set of components of $Z$.
\end{lem}
\begin{proof} 
The obstructions to the surjectivity of $X^F\to Y^F$ live in the set $H^1(F,Z)$ of $F$-conjugacy classes in $Z$. By Lang's theorem, we have $H^1(F,Z)\cong H^1(F,\pi_0(Z))$, where $\pi_0(Z)$ denotes the component group of $Z$. It is clear that $\#Y^F/\#I_f\leq \#H^{1}(F,\pi_0(Z))\leq \#\pi_0(Z)=K_f$. This completes the proof.
\end{proof}

We can now state the main results of this section.

\begin{prop}\label{pro-principal-efficiency}
    Suppose that the almost multiplicity-one property holds for $H$ with respect to principal-series representations. Then:
    \begin{itemize}
            \item[i)] There exists a unique $w\in B(\mathrm k)\backslash G(\mathrm k)/H(\mathrm k)$ such that $\mathrm d(B\cap w Hw^{-1})$ is a finite algebraic group: that is, the scheme $\mathrm d(B\cap w Hw^{-1})$ is finite over $\mathrm k$.
            \item[ii)] Let $w\in B(\mathrm k)\backslash G(\mathrm k)/H(\mathrm k)$ be the element specified in i). The group $B\cap w Hw^{-1}$ endowed with the reduced scheme structure is a connected unipotent algebraic group. 
            \end{itemize}   
\end{prop}

\begin{proof}
    The proof is essentially an application of the classical Mackey formula, as can be seen as follows. In this proof,  for positive integers $n$, we systematically view characters of $T^{F^n}$ as representations of $B^{F^n}$ via the natural quotient $B^{F^n}\to T^{F^n}$.

    Fix any $v\in B(\mathrm k)\backslash G(\mathrm k)/H(\mathrm k)$, then $v$ has a representative $v_0$  in $G^{F^N}$ for some integer $N$.  The Mackey formula claims that for any positive integer $n$ and any  character $\chi:T^{F^n}\to \Qlb^\times$, we have 
    $$
    \langle R_{T,\chi}^{(n)},1_{H^{F^n}}\rangle_{H^{F^n}}=
    \langle \mathrm{Ind}_{B^{F^n}}^{G^{F^n}}(\chi),1_{H^{F^n}}\rangle_{H^{F^n}}=
    \sum\limits_{w\in B^{F^n}\backslash G^{F^n}/H^{F^n}}\langle \chi_{|B^{F^n}\cap w H^{F^n}w^{-1}},1_{B^{F^n}\cap w H^{F^n}w^{-1}}\rangle_{B^{F^n}\cap w H^{F^n}w^{-1}}.
    $$
    Set temporarily $S_v:=\mathrm d(B\cap v_0Hv^{-1}_0)$. Let $C_v$ be the component group of the kernel of the quotient map $$B\cap v_0Hv^{-1}_0\rightarrow S_v.$$ For integers $n$ that are divisible by $N$,  Lemma \ref{lem-imageFpoints} shows that there are exactly $K_n^{v_0}$  many characters $\chi:T^{F^n}\to \Qlb^\times $  with  $$ \langle \chi_{|B^{F^n}\cap v_0 H^{F^n}v^{-1}_0},1_{B^{F^n}\cap v_0 H^{F^n}v^{-1}_0}\rangle_{B^{F^n}\cap v_0 H^{F^n}v^{-1}_0}=1,$$ for a specific integer $K_n^{v_0}$ satisfying $\frac{\#T^{F^n}}{\#S_v^{F^n}} \leq K_n^{v_0}\leq \frac{\#T^{F^n}\cdot \#C_v}{\#S_v^{F^n}}$. (Here, we apply Lemma \ref{lem-imageFpoints} to the map $B\cap v_0 Hv_0^{-1}\to S_v$ induced by $\mathrm d$, and view it as a morphism defined over $\mathbb F_{q^n}$.)


Combining the above computation and  the almost multiplicity-one property for $H$ with respect to principal-series representations, we \textbf{claim}:
\begin{itemize}
    \item There exists a unique $w\in B(\mathrm k)\backslash G(\mathrm k)/H(\mathrm k)$ such that $\mathrm d(B\cap w Hw^{-1})$ is a finite scheme over $\mathrm k$.
\end{itemize}

Indeed, if for every $w\in B(\mathrm k)\backslash G(\mathrm k)/H(\mathrm k)$, the scheme $\mathrm d(B\cap w Hw^{-1})$ has positive dimension, then $
    \lim\limits_{n\to \infty} \frac{\#\mathcal P^{\diamondsuit}(G,H,n)}{\#\mathcal P(G,n)}=0
    $
by the above computation and Remark \ref{rm-rationorbit-vs-geoorbit}. 

If there are distinct $w_1,w_2\in B(\mathrm k)\backslash G(\mathrm k)/H(\mathrm k)$ such that $S_1:=\mathrm d(B\cap w_1 Hw^{-1}_1)$ and $S_2:=\mathrm d(B\cap w_2 Hw^{-1}_2)$ are both finite as schemes over $\mathrm k$, then the algebraic subgroup $S$ of $T$ generated by $S_1$ and $S_2$ is likewise finite. By the above computation, there exist at least $\frac{\#T^{F^n}}{\#S^{F^n}}$ many characters $\chi$ such that $\langle \mathrm{Ind}_{B^{F^n}}^{G^{F^n}}(\chi),1_{H^{F^n}}\rangle_{H^{F^n}}\geq 2$ for sufficiently divisible $n$. This violates the almost multiplicity-one property for $H$ with respect to principal-series representations. And we finish the proof of the above claim.

In what follows, we fix this unique $w$ specified in the above claim (and fix a representative $w_0$ in $G(\mathrm k)$ ). We set $B_w:=B\cap w_0Hw^{-1}_0$ and $S_w:=\mathrm d(B_w)$. It remains to show that $B_w$ is connected, or (equivalently) that $\pi_0(B_w)$ is trivial. By Remark \ref{rm-rationorbit-vs-geoorbit}, we see that for sufficiently divisible positive integers $n$, the set $\Phi_{q^n}^{-1}\left(\{w\}\right)$ is {\bf not} a singleton set if $\pi_0(B_w)$ is nontrivial. 

Suppose for the contrary that $\pi_0(B_w)$ is nontrivial. For integers $n$ as in the last paragraph, let $w_1,w_2\in B^{F^n}\backslash G^{F^n}/H^{F^n} $ be distinct with $\Phi_{q^n}(w_1)=\Phi_{q^n}(w_2)=w$. By Mackey formula, we have
\begin{align*}
&\langle \mathrm{Ind}_{B^{F^n}}^{G^{F^n}}(\chi),1_{H^{F^n}}\rangle_{H^{F^n}}\geq\\ &\langle \chi_{|B^{F^n}\cap w_1 H^{F^n}w_1^{-1}},1_{B^{F^n}\cap w_1 H^{F^n}w_1^{-1}}\rangle_{B^{F^n}\cap w_1 H^{F^n}w_1^{-1}}+\langle \chi_{|B^{F^n}\cap w_2 H^{F^n}w_2^{-1}},1_{B^{F^n}\cap w_2 H^{F^n}w_2^{-1}}\rangle_{B^{F^n}\cap w_2 H^{F^n}w_2^{-1}}
\end{align*}
for characters $\chi:T^{F^n}\to \Qlb^\times$. 
Hence there exist at least $\frac{\#T^{F^n}}{\#S_w^{F^n}}$ many characters $\chi$ satisfying $$\langle \mathrm{Ind}_{B^{F^n}}^{G^{F^n}}(\chi),1_{H^{F^n}}\rangle_{H^{F^n}}\geq 2,$$ which violates the almost multiplicity-one property with respect to principal-series representations.
\end{proof}

\begin{cor}\label{cor-efficient}
    Suppose that the almost multiplicity-one property holds for the pair $(G,H)$. Then:
    \begin{itemize}
            \item[i)] There exists a unique $w\in B(\mathrm k)\backslash G(\mathrm k)/H(\mathrm k)$ such that $\mathrm d(B\cap w Hw^{-1})$ is a finite algebraic group: that is, the scheme $\mathrm d(B\cap w Hw^{-1})$ is finite over $\mathrm k$;
            \item[ii)] Let $w\in B(\mathrm k)\backslash G(\mathrm k)/H(\mathrm k)$ be the element specified in $i)$. The group $B\cap w Hw^{-1}$ endowed with the reduced scheme structure is a connected unipotent algebraic group. 
            \end{itemize}
\end{cor}
\begin{proof}
This follows from Corollary \ref{cor-regularprincipal}, Proposition \ref{pro-principal-efficiency} and the fact that
$$
   \lim_{n\to \infty} \frac{ \#\{\mathrm {Ind}_{B^{F^n}}^{G^{F^n}}\chi:\text{$\chi$ is regular} \}}{\#\mathcal P(G,n)}=1.
   $$

\end{proof}
 
\section{Cohomological Interpretation}\label{sec-cohomological-interpret}
We introduce the notation adopted throughout this section.
For a scheme $X_0$ over $\mathbb F_q$, denote by $X$  its pullback to $\mathrm k$. For Weil sheaves $\mathscr L$ on $X$ and positive integers $n$, let $K^n_{\mathscr L}$ denote the function $X^{F^n}\to \Qlb$ given by ``trace of Frobenius" (also known as the sheaf-function correspondence in Grothendieck's sense). More generally, for a complex $\mathscr K$ of Weil sheaves on $X$, let $K^n_{\mathscr K}$ denote the function $\sum_{i\in \mathbb Z}(-1)^iK^n_{\mathscr H^{i}(\mathscr K)}$. For a morphism $f:X\to Y$ of schemes, all functors $f^!,f_!,f_*,f^*$ are understood in the derived sense.

\subsection{Sheaves on tori}In this subsection, we recall some basic facts about $1$-rank tame local systems on a torus.

Here we introduce the notation. In this subsection, let $T_0$ be a torus over $\mathbb F_q$ and let $T$ be its pullback to $\mathrm k$. We define $(T^F)^\vee$ to  be the set of characters of $T^F$ with values in $\Qlb^\times$.

In the following remark, we recall \cite[1.3 and 1.6]{L5}. An equivalent formulation can be found in \cite[2.2]{Lucha}.
\begin{rmk}\label{rm-kummersheaf} 
    For a torus $S$ over $\mathrm k$,
    set $\mathfrak L(S)$ to be the isomorphism classes of local systems $\mathscr L$ such that $\rank (\mathscr L)=1$ and $\mathscr L^{\otimes n}\cong \Qlb$ for some integer $n\geq 1$, invertible in $\mathrm k$. We fix an isomorphism of abstract groups $\mathfrak L(\mathrm G_m)\cong \mathrm k^\times$ once and for all, where $\mathrm G_m$ is the $1$-dimensional split torus. This gives rise to an isomorphism $\lambda_S:\mathfrak{ L}(S)\xrightarrow{\sim} S^*(\mathrm k)$, where $S^*$ denotes the dual torus of $S$.  For a morphism $h:S'\to S''$ of tori, we have the following commutative diagram $$
    \xymatrix{
  &\mathfrak{L}(S '')\ar[r]^{\lambda_{S''}}\ar[d]&(S'')^*(\mathrm k)\ar[d]\\
    &\mathfrak{L}(S')\ar[r]^{\lambda_{S'}}& (S')^*(\mathrm k)
    }
    $$
    where the right vertical map is the transpose of $h$, and the left vertical map is given by $h^*$. When $S$ is defined over $\mathbb F_q$ with the geometric Frobenius $F$, then so is $S^*$ and  the map $\lambda_S$ defines a bijection $\mathfrak L(S)^F\xrightarrow{\sim} (S^*)^F=(S^F)^\vee$, where $\mathfrak{ L}(S)^F:=\{\mathscr L\in \mathfrak{L}(S):F^* \mathscr L\cong \mathscr L\}$.
\end{rmk}

In the following remark, we explicitly construct local systems that encode characters of the $F$-fixed point group of a torus.

\begin{rmk}\label{rm-torishf}
    We have a homomorphism $\mathrm l^T:T\to T$ defined by $t\mapsto F(t)\cdot t^{-1}$. 
    We verify that $\mathrm l^T$ is  finite \'etale and compatible with the geometric Frobenius endomorphism $F$. We have the decomposition of Weil sheaves $$
   \mathrm l^T_* \Qlb = \mathrm l^T_! \Qlb \cong \bigoplus_{\eta\in (T^F)^\vee} \mathscr L_\eta^T,
    $$
    where $\mathscr L_\eta^T$ is the $1$-rank local system on $T$ satisfying the following property: (We also denote $\mathscr L_\eta=\mathscr L_\eta^T$ if it causes no confusion.)
    \begin{itemize}
        \item For any positive integer $n$, the corresponding function $K_{\mathscr L _\eta}^n: T^{F^n}\to \Qlb$  is explicitly given by the assignment
        $$
        T^{F^n}\ni t\mapsto \eta \circ \mathrm N_T^n(t),
        $$
        where $\mathrm N_T^n:T^{F^n}\to T^F$ is the norm map given by $T^{F^n}\ni t\mapsto t\cdot F(t) \cdot \ldots \cdot F^{n-1}(t)$.
    \end{itemize}

 Note that each $\mathscr L_\eta$ is indeed a local system introduced in \cite[1.3 and 1.6]{L5} endowed with a ``natural normalization'' in the sense of \select{loc. cit.}. Consequently, we see that $\{\mathscr L_\eta:\eta\in (T^F)^\vee\}$ gives a set of representatives of $\mathfrak L(T)^F$ (introduced in Remark \ref{rm-kummersheaf}). (We have $(\mathscr L_\eta)^{\otimes (\#T^F)}\cong \Qlb$, and the integer $\#T^F$ is coprime to $p$, the radical of $q$.)
\end{rmk}

For future use, we record the following remark which is a straightforward variant of Remark \ref{rm-torishf}.

\begin{rmk}\label{rm-gentorshf}
    More generally, fix a positive integer $n$ and an $F^n$-stable  subtorus $S$ of $G$. We have a homomorphism $\mathrm l_n^S:S\to S$ defined by $s\mapsto F^{n}(s)\cdot s^{-1}$.  Similarly to Remark \ref{rm-torishf}, we have
    $$
    (\mathrm l_n^S)_* \Qlb\cong \bigoplus_{\eta \in (S^{F^n})^\vee}\mathscr L_{\eta}^S.
    $$
    Here, the  function corresponds to $\mathscr L_\eta^S$ is (for $N$ divisible by $n$)
    $$
    S^{F^N}\ni s\mapsto \eta\circ \mathrm{N}_S^{N,n}(s),
    $$
    where $\mathrm N_{S}^{N,n}:S^{F^N}\to S^{F^n}$ is the norm map.
\end{rmk}

The following lemma shows that the family of functions determines the corresponding local system.

\begin{lem}\label{lem-chebo-tori}
Fix $\eta\in (T^F)^\vee$ and a $1$-rank lisse Weil sheaf $\mathscr L$ on $T$. Suppose that for every positive integer $n$, the corresponding function $K_{\mathscr L}^n$ equals the function $K^n_{\mathscr L_\eta}$ introduced in Remark \ref{rm-torishf}, then we have $\mathscr L\cong \mathscr L_\eta$ as Weil sheaves.
\end{lem}
\begin{proof}
    This should follow from an appropriate version of Chebotarev density theorem. In what follows we prove the current lemma, using Deligne's weight theory.

    Set $\mathscr L_0=\mathscr L^{-1}\otimes \mathscr L_{\eta}$. Then the corresponding function $K^n_{\mathscr L_0}: T^{F^n}\to \Qlb$ is the constant function with the value $1$. By Grothendieck trace formula, we have \begin{equation}\label{eq-lem-count}
        \mathrm {Tr}\left(({F^*})^n,H^\bullet_c(T,\mathscr L_0)\right)=|T^{F^n}|.
    \end{equation}

    It is clear that $H^{i}_c(T,\mathscr L_0)=0$ unless $0\leq i\leq 2\dim (T)$. Moreover, for eigenvalues $\alpha$ of $F^*$ on the  space $H^i_c(T,\mathscr L_0)$, we have $|\alpha|\leq q^{i/2}$ by Deligne's weight theory \cite[Th\'eor\`eme 3.3.1]{D}. Consequently, by Equation (\ref{eq-lem-count}) (let $n\to \infty$), we have $H^{2\dim (T)}_c(T,\mathscr L_0)\neq 0$. Hence we have $H^0(T,\mathscr L_0^{-1})\neq 0$ by Poincar\'e duality, giving rise to a nontrivial homomorphism (hence is an isomorphism) $\mathrm I :\mathscr L_\eta\to \mathscr L$. 
   Observe that $1=K_{\mathscr L}^1(e)={K}^1_{\mathscr L_\eta}(e)$, rigidifying the isomorphism $\mathrm I$ as an isomorphism of Weil sheaves. This completes the proof.
\end{proof}
For future use, we record the following lemma.
\begin{lem}\label{lem-vanis-tori}
    Suppose that $\eta^0\in (T^F)^\vee$ is nontrivial. Then  $H^i_c(T,\mathscr L_{\eta^0})=0$ for all integers $i$.
\end{lem}
\begin{proof}
    In view of the decomposition $$
    \mathrm l^T_! \Qlb \cong \bigoplus_{\eta\in (T^F)^\vee} \mathscr L_\eta
    $$
    introduced in Remark \ref{rm-torishf}, we see that 
    \begin{equation}\label{eq-lem-vanis}
         H^i_c(T,\Qlb)\cong \bigoplus_{\eta\in (T^F)^\vee} H^i_c(T,\mathscr L_\eta)
    \end{equation}
   for all integers $i$.
   
    Note that for $\eta=1$ the trivial character, we have $\mathscr L_{1}\cong\Qlb$ the constant sheaf of rank $1$. We finish the proof by comparing dimensions of both sides of Equation (\ref{eq-lem-vanis}).
 \end{proof}

\begin{rmk}\label{rm-vanishin-tori}
Fix a non-constant $\mathscr L\in \mathfrak{L}(T)$.
    Indeed, we can show  $H^i_c(T,\mathscr L)=0$ for all integers $i$. See \cite[(1.11.1)]{Lucha}.
\end{rmk}

\subsection{Character sheaves and Deligne-Lusztig characters}\label{sec-chashandDL}
In this subsection we recall \cite{L5} and \cite{Lau}.

For a Borel pair $(B,T)$ with $F$-stable $T$, set $\mathrm d^B_T:B\to T$ to be the map witnessing $T$ as the reductive quotient and providing a section for the inclusion $T\hookrightarrow B$.
We introduce the following commutative diagram and fix the notation.
\begin{equation}\label{diag-charshf}
\xymatrix{
&T_{reg}\ar[d]_{j_T}&\wt{G}_{rss}\ar[l]_{\rho_{rss}}\ar[d]^{\wt j}\ar[r]^{\pi_{rss}}&G_{rss}\ar[d]_j\\
&T &\wt{G}\ar[l]_{\rho}\ar[r]^{\pi}&G
}    
\end{equation}
where \begin{itemize}
    \item $T_{reg}:=\{t\in T: (C_G(t))^\circ=T\}$;
    \item $G_{rss}=\bigcup_{h\in G}h T_{reg}h^{-1}$;
    \item $\wt G_{rss}:=\{(g,hT)\in G_{rss}\times G/T:h^{-1}gh\in T_{reg}\}$;
    \item $\wt G:=\{(g,hB)\in G\times G/B: h^{-1}gh\in B\}$;
    \item $\rho_{rss}$ sends $(g,hT)\in \wt G_{rss}$ to $h^{-1}gh$;
    \item $\pi_{rss}$ sends $(g, hT)\in \wt G_{rss}$ to $g$;
    \item $\rho$ sends $(g,hB)\in \wt G$ to $\mathrm d^B_T(hgh^{-1})$;
    \item $\pi$ sends $(g,hB)\in \wt G$ to $g$;
    \item $\wt j$ sends $(g,hT)\in \wt G_{rss}$ to $(g,hB)\in \wt G$;
    \item $j_T$ and $j$ are natural inclusions.
    \end{itemize}

\begin{defn}\label{def-scRTchi}
    For a character $\eta:T^F\to \Qlb^\times$, recall the sheaf $\mathscr L_\eta$ introduced in Remark \ref{rm-torishf}. We set $$\mathscr R_{T,\eta}:= \left(j_{!*}(\pi_{rss})_!\rho^*_{rss} j_{T}^{*}\mathscr L_\eta [\dim G]\right)[-\dim G],$$ where $j_{!*}$ is the intermediate extension along $j$. 
\end{defn}
We note that the sheaf $\mathscr L_\eta$ and the upper row of the Diagram (\ref{diag-charshf}) is indeed defined over $\mathbb F_q$. Hence $\mathscr R_{T,\eta}$ has a Weil structure $\phi_{T,\eta}: F^* \mathscr R_{T,\eta}\xrightarrow{\sim} \mathscr R_{T,\eta}$ arising from the intermediate extension $j_{!*}$.

\begin{prop}\label{prop-charactersheaf-main}
    The following statements hold:
    \begin{itemize}
        \item [(i)] Suppose that the hypotheses of \cite[Proposition 8.15]{L5} hold for the pair $(G,T)$. Then the function $K^n_{\mathscr R_{T,\eta}}$ corresponding to $\mathscr R_{T,\eta}$ is the Deligne-Lusztig character $R_{T,\eta\circ \mathrm N_T^n}^{(n)}:G^{F^n}\to \Qlb$ (see the sentence before Definition \ref{de-regratio} for the notation), where $\mathrm N_T^n:T^{F^n}\to T^F$ is the norm map;
        \item [(ii)] There is a canonical isomorphism $\mathscr R_{T,\eta}\cong \pi_! \rho^* \mathscr L_\eta$ of $\Qlb$-sheaves.
    \end{itemize}
\end{prop}
\begin{proof}
    (i) follows from \cite[Corollaire 2.3.2]{Lau}. (ii) is \cite[Th\'eor\`eme 1.2.2]{Lau}. Note that the complex $\mathscr R_{T,\eta}$ is the $[-\dim G]$ shift of  the corresponding perverse sheaf in  \select{loc. cit.}.
\end{proof}

The following two remarks address the hypotheses of \cite[Proposition 8.15]{L5}, but Remark \ref{rm-*qsize} will not be used further in this paper.

\begin{rmk}\label{rm-qsize}
    In view of \cite[Proposition 8.15]{L5}, we verify that there exists an integer $N_G$ depending only on the root datum of $G$, satisfying the following property:
    \begin{itemize}
        \item Fix a positive integer $n$ and an $F^{n}$-stable maximal torus $S$ of $G$. Suppose that $q^n\geq N_G$, then the hypotheses of \cite[Proposition 8.15]{L5} hold for the pair $(G,S)$ (viewed as a pair defined over $\mathbb F_{q^n}$).
    \end{itemize}
    (We may use \cite[1.6 and Lemma 8.2]{L5} to see it.)
\end{rmk}

\begin{rmk}\label{rm-*qsize}
    By \cite[(1.7.1) and (1.9.1)]{Sho} and \cite[Theorem 5.5]{Sho2}, Proposition \ref{prop-charactersheaf-main} (i) holds without the assumption ``suppose that the hypotheses of \cite[Proposition 8.15]{L5} hold for the pair $(G,T)$".  But the authors ignore this fact while writing this paper. And we will not use this remark in the rest of this paper.
\end{rmk}

Suppose (in this paragraph) that both $T$ and $B$ are $F$-stable, then the whole diagram (\ref{diag-charshf}) is defined over $\mathbb F_q$. Consequently, for characters $\eta:T^F\to \Qlb^\times$, we have a Weil structure $\phi^\circ_{T,\eta}:F^* \left(  \pi_!\rho^* \mathscr L_\eta   \right)\xrightarrow{\sim} \pi_! \rho^* \mathscr L_\eta$ given by Grothedieck's construction.

\begin{cor}\label{cor-FstableBorel}
    Suppose that the Borel pair $(B,T)$ is $F$-stable (i.e., both $T$ and $B$ are $F$-stable), then the Weil structure $\phi_{T,\eta}^\circ$ defined in the last paragraph coincides with the Weil structure $\phi_{T,\eta}$ under the identification introduced in Proposition \ref{prop-charactersheaf-main} (ii).
\end{cor}
\begin{proof}
    It suffices to show that $\phi_{T,\eta}^\circ[\dim G]$ coincides with $\phi_{T,\eta}[\dim G]$ under the (shifted) identification introduced in Proposition \ref{prop-charactersheaf-main} (ii).  
    
    By the formal property of intermediate extension, it suffices to see $\phi_{T,\eta}^\circ[\dim G]$ coincides with $\phi_{T,\eta}[\dim G]$ when restricted to $G_{rss}$. Since both $\phi_{T,\eta}^\circ[\dim G]|G_{rss}$ and $\phi_{T,\eta}[\dim G]|G_{rss}$ arise  from Grothendieck's construction, our assertion follows.
\end{proof}
\begin{rmk}
    Corollary \ref{cor-FstableBorel} imposes no condition on the size of $q$.
\end{rmk}
\subsection{Periods and sheaves}In this subsection, we will show that certain top cohomology spaces capture our desired  multiplicities.

 
\begin{defn}\label{def-periods}
    Fix an $F$-stable maximal torus $T$ of $G$ and a character $\chi:T^F\to \Qlb^\times$. For positive integers $n$, we set 
    $$
    \mathrm P_{T,\chi}(n):=\langle R_{T,\chi\circ \mathrm N_T^n}^{(n)},1_{H^{F^n}}\rangle_{H^{F^n}},
    $$
    where $\mathrm N_T^n:T^{F^n}\to T^F$ is the norm map.
\end{defn}
Set $\mathcal X:=G/H$ and $\wt G_H:=\{(g_1,g_2H)\in G\times \mathcal X:g_2^{-1}g_1g_2\in H\}$. Let $\tau: \wt G_H\to G$ be the projection to the first factor. Set $\mathscr I_{H}:=\tau_! \Qlb$. We endow $\mathscr I_H$ with the obvious Weil structure via Grothendieck's construction.
Then the function $K_{\mathscr I_H}^n$ corresponding to $\mathscr I_H$ is the character of $\mathrm {Ind}_{H^{F^n}}^{G^{F^n}}(1_{H^{F^n}})$ by Grothendieck trace formula.

The following proposition is a trivial instance of  Grothendieck  trace formula. Similar argument is adopted in \cite{S}.
\begin{prop}\label{prop-periodtraceformula}
Let the notation be as in Definition \ref{def-periods}. Fix a positive integer $n$.  Suppose the hypotheses of \cite[Proposition 8.15]{L5} hold for the pair $(G,T)$ (viewed as a pair defined over $\mathbb F_{q^n}$). Then we have the following equation 
    $$
  \#G^{F^n} \cdot  \mathrm P_{T,\chi}(n)=\mathrm {Tr}((F^n)^*,H^\bullet_c(G,\mathscr R_{T,\chi}\otimes \mathscr I_H)).
    $$
\end{prop}

\begin{defn}\label{def-scheme-badmodel}
    For a Borel pair $(B,T)$ of $G$, we set $$\mathcal Y_{T,B}^H:=\{(g,xB,yH)\in G\times G/B\times G/H:\text{$x^{-1}gx\in B$ and $y^{-1}gy\in H$}\}.$$
    Define a map $\rho_{T,B,H}:\mathcal Y_{T,B}^H\to T$ by sending $(g,xB,yH)\in \mathcal Y_{T,B}^H$ to $\mathrm d_T^B(x^{-1}gx)$. Let $\tau_{T,B}^H:\mathcal Y_{T,B}^H\to G$ be the projection to the first factor.
\end{defn}

If the Borel pair $(B,T)$ is $F$-stable (we assume so in this paragraph), then $\mathcal Y_{T,B}^H$ is the pullback to $\mathrm k$ of the corresponding scheme over $\mathbb F_q$. Consequently, for any character $\chi:T^F\to \Qlb^\times$, the sheaf $\rho_{T,B,H}^* \mathscr L_\chi$ has been endowed with a natural Weil structure via the Grothendieck's construction.
\begin{cor}\label{cor-Fstableperiods}
    Fix an $F$-stable Borel pair $(B,T)$ of $G$. For all characters $\chi$ of $T^F$, we have a canonical identification of Weil sheaves $\mathscr R_{T,\chi}\otimes \mathscr I_H\cong (\tau_{T,B}^H)_! \rho_{T,B,H}^*\mathscr L_\chi .$ 
   In particular, we have the following equation
    $$
   \#G^{F^n} \cdot  \mathrm P_{T,\chi}(n)=\mathrm {Tr}\left((F^n)^*,H^\bullet_c(\mathcal Y_{T,B}^H,\rho_{T,B,H}^*\mathscr L_\chi)\right),
    $$
    provided that  the hypotheses of \cite[Proposition 8.15]{L5} hold for the pair $(G,T)$.
\end{cor}

\begin{proof}
Note that we have the following cartesian square: (Here $\pi$ is the morphism $\pi$ introduced in Section \ref{sec-chashandDL} with respect to the Borel pair $(B,T)$.)

$$
\xymatrix{
&\mathcal Y_{T,B}^H \ar[d]\ar[r]&\wt G_H\ar[d]^\tau\\
&\wt G\ar[r]^\pi&G
}
$$
    Our assertion follows from Proposition \ref{prop-periodtraceformula}, Corollary \ref{cor-FstableBorel}, proper base change and the Grothendieck trace formula.
\end{proof}

\begin{rmk}\label{rm-badmodel}
    Fix a Borel pair $(B,T)$ of $G$ with $F$-stable $T$. Then we have an identification of $\Qlb$-sheaves $\mathscr R_{T,\chi}\otimes \mathscr I_H\cong (\tau_{T,B}^H)_! \rho_{T,B,H}^* \mathscr L_\chi $. This follows from proper base change.  
\end{rmk}

\begin{cor}\label{cor-cohomoweight}
  Fix a Borel pair $(B,T)$ of $G$ with $F$-stable $T$.  Fix an integer $i$. For eigenvalues $\alpha$ of $F^*$ on the cohomology space $H^i_c(G,\mathscr R_{T,\chi}\otimes \mathscr I_H)$, we have the following: 
  \begin{itemize}
      \item [(i)] $|\alpha|\leq q^{i/2}$; (Here we fix an arbitrary identification $\Qlb\cong \mathbb C$ of fields.)
      \item [(ii)] Suppose that $i=2\dim G$, then $\alpha/q^{\dim G}$ is a root of unity.
  \end{itemize}
\end{cor}

\begin{proof}
    In proving this corollary, we may replace the field $\mathbb F_q$ by an extension of sufficiently divisible degree (and the Frobenius is replaced by the corresponding power). Consequently, we may assume that $T$ is contained in an $F$-stable Borel subgroup $B$ of $G$, and that the hypotheses related to \cite[Proposition 8.15]{L5} are fulfilled. And we assume this henceforth.

Using Corollary \ref{cor-Fstableperiods}, we identify the cohomology space $H^i_c(G,\mathscr R_{T,\chi}\otimes \mathscr I_H)$ with $H^i_c(\mathcal Y_{T,B}^H,\rho_{T,B,H}^*\mathscr L_\chi)$. This identification is compatible with the endomorphisms $F^*$.
Note that $\rho_{T,B,H}^* \mathscr L_\chi$ is pure of weight $0$ as a sheaf on $\mathcal Y_{T,B}^H$ with respect to the natural Weil structure.
Then assertion (i) follows from  Deligne's weight theory \cite[Th\'eor\`eme 3.3.1]{D} (applied to $H^i_c(\mathcal Y_{T,B}^H,\rho_{T,B,H}^*\mathscr L_\chi)$). And assertion (ii) follows from  Poincar\'e duality (see Lemma \ref{lem-tpduality} below) and the fact that $\dim \mathcal Y_{T,B}^H=\dim G$ (see Proposition \ref{pro-dimYTBH}).
\end{proof}

In the above proof, we use the following lemma.
\begin{lem}\label{lem-tpduality}
    Let $X_0$ be an algebraic scheme of dimension $\leq n$ over $\mathbb F_q$. 
    Let $X$ be the pullback of $X_0$ to $\mathrm k$. Let $\mathscr F$ be a Weil sheaf on $X$ that is pure of weight $0$ such that all characteristic functions $K_\mathscr F^m$ (for all positive integers $m$) take values in roots of unity. Suppose that $\alpha$ is an eigenvalue of $F^*$ on $H^{2 n}_c(X,\mathscr F)$, then the number $\alpha/q^{ n}$ is a root of unity.
\end{lem} 
\begin{proof}
    We may replace $X_0$ by its reduced subscheme with the same underlying topological space. Hence we assume $X_0$ is reduced. We choose a open subscheme $V_0$ of $X_0$ such that (Let $V$ be the pullback of $V_0$ to $\mathrm k$ and let $j:V\hookrightarrow X$ be the inclusion.)
    \begin{itemize}
        \item $V_0$ contains the generic points of all irreducible components  of $X_0$ that is of dimension $n$. 
        \item $V_0$ is smooth over $\mathbb F_q$.
        \item The restriction of $\mathscr F$ to $V$ is lisse.
    \end{itemize}
    We verify that such $V_0$ exists. Let $i:Z\hookrightarrow X$ be the complement of $j$, with $Z$ reduced. Note that $\dim Z\leq n-1$.
     We have the distinguished triangle
     $$
     j_!j^* \mathscr F\to \mathscr F\to i_*i^*\mathscr F\to.
     $$
     This gives rise to an identification $H^{2n}_c(V,j^*\mathscr F)\cong H^{2n}_c(X,\mathscr F)$ that is compatible with the action of the Frobenius. Note that the Frobenius permute the irreducible components of $V$. Since the characteristic functions of $\mathscr F$ take values in roots of unity, it is easy to see that the eigenvalues of the operator $F^*$ on $H^0(V,(j^*\mathscr F)^\vee)$ are all roots of unity, where $(j^*\mathscr F)^\vee$ denote the lisse Weil sheaf on $V$ such that it is dual to the restriction of $j^*\mathscr F$ over each irreducible component of $V$.  Using Poincar\'e duality, we conclude the proof.
\end{proof}

\begin{rmk}
   Alternatively, the above corollary follows from the ``semicontinuity of weights" with respect to the intermediate extension.
\end{rmk}

Let $\mathbb Z_+$ denote the set of positive integers.
We introduce the following definition to invoke Lemma \ref{lem-const}.
\begin{defn} \label{def-gtype}
Let  $\CP\subset \BZ_+$ be an arithmetic progression. A function $M: \CP \to \BC $ is said to be of geometric type if it is of the form
\[
M(\nu) = \frac{\sum^k_{i=1} a_i \alpha_i^\nu}{\sum^l_{j=1} b_j \beta_j^\nu } ,\quad \nu\in\CP,
\]
where $a_i, \alpha_i, b_j, \beta_j \in\BC$, and the denominator is nonzero for every $\nu\in \CP$. If the denominator is a nonzero constant, we say that $M$ is of trace type.  
\end{defn}
In the remainder of this paper, we fix an identification $\Qlb\cong \mathbb C$.
We have the following elementary lemma, see  \cite[Lemma 2.2]{LMS}. 

\begin{lem} \label{lem-const}
Let $M$ be a function of geometric type defined on an arithmetic progression $\CP\subset \mathbb Z_+$. If $M$ is integer-valued and  has a finite limit $L\in \mathbb C$ as $\nu\to\infty$ through $\CP$, then 
$M$ is a constant function taking the value $L$. 
\end{lem}

We can  state the main theorem of this subsection.

\begin{thm}\label{thm-pertopcohomo}
  Let the notation be as in Definition \ref{def-periods}. Fix a positive integer $n$.   Suppose that $q^n\geq N_G$, where $N_G$ is the integer introduced in Remark \ref{rm-qsize}.  Then  we have $$
    \mathrm P_{T,\chi}(n)=\frac{1}{q^{n\dim (G)}}\mathrm {Tr}\left((F^n)^*, H^{2\dim G}_c(G,\mathscr R_{T,\chi}\otimes \mathscr I_H)\right ).
    $$
\end{thm}

\begin{proof}
   This is a combination of  Corollary \ref{cor-cohomoweight} and Lemma \ref{lem-const}. 

    In what follows we elaborate. We have $H^{i}_c(G,\mathscr R_{T,\chi}\otimes \mathscr I_H)=0$ unless $0\leq i\leq 2\dim G$ by Remark \ref{rm-badmodel} and the fact that $\dim \mathcal Y_{T,B}^H=\dim G$ (see Proposition \ref{pro-dimYTBH}). By Corollary \ref{cor-cohomoweight}, we can choose a sufficiently divisible integer $N\geq 1$ such that $(F^N)^*$ acts on $H^{2\dim G}_c(G,\mathscr R_{T,\chi}\otimes \mathscr I_H)$ as the multiplication by $q^{N\dim G}$. Hence by Proposition \ref{prop-periodtraceformula}, Remark \ref{rm-qsize} and Corollary \ref{cor-cohomoweight}, we have (note that by Remark \ref{rm-qsize}, the hypotheses of  \cite[Proposition 8.15]{L5} are fulfilled)
$$
\lim\limits_{ k\to \infty}\mathrm P_{T,\chi}(n+Nk)= \frac{1}{q^{n\dim (G)}} \mathrm {Tr}\left((F^n)^*, H^{2\dim G}_c(G,\mathscr R_{T,\chi}\otimes \mathscr I_H)\right ).
$$    
By Lemma \ref{lem-const}, we have
$$
\mathrm P_{T,\chi}(n+Nk)= \frac{1}{q^{n\dim (G)}}\mathrm {Tr}\left((F^n)^*, H^{2\dim G}_c(G,\mathscr {R}_{T,\chi}\otimes \mathscr I_H)\right),
$$
for all nonnegative integers $k$.
(Note that $\mathrm P_{T,\chi}$ is integer-valued and of geometric type in the sense of Definition \ref{def-gtype}, due to Proposition \ref{prop-periodtraceformula}.)
\end{proof}

\section{The Schemes $\mathcal X$ and $\mathcal Y_{T,B}^H$}\label{sec-thegeometry}
In this section, we first recall some facts introduced in \cite{Kno}. Then we show in Theorem \ref{thm-almostmulone} that the condition $\star$ in Remark \ref{rm-changeBcondition} implies the almost multiplicity-one property. After some preliminaries on Steinberg representations, we prove the main theorem and conclude with Theorem \ref{thm-main}.

As in the previous section, we denote $\mathcal X:=G/H$. For a Borel pair $(B,T)$ of $G$ with $F$-stable $T$, recall the scheme $\mathcal Y_{T,B}^H$ introduced in Definition \ref{def-scheme-badmodel}.

\subsection{$B$-stabilizers}
In this subsection we fix a Borel pair $(B,T)$ of $G$. Since we force the variety $\mathcal X=G/H$ to be spherical, there exists a unique open dense $B$-orbit  on $\mathcal X$, which we denote by $\mathcal O_\mathcal X^B$. We work with schemes over $\mathrm k$ in this subsection.

As in \cite[Section 2]{Kno}, we adopt the following definition.
\begin{defn}\label{def-rankspherical}
    For a $B$-variety $Z$, we define the following invariance: (see \select{loc. cit.} for details.)
    \begin{itemize}
        \item $\chi(Z):=\{\chi_f\in \chi(B):f\in \mathrm k(Z)^{(B)}\}$; (the character group of $Z$)
        \item $\mathrm {rk}(Z):=\rank (\chi (Z))$.
    \end{itemize}
\end{defn}

Also, for a $G$-variety $X$, we define $\mathfrak B(X)$ to be the set of non-empty, closed, irreducible $B$-stable subsets of $X$.

The following lemma is a trivial instance of \cite[Corollary 2.4]{Kno}.
\begin{lem}\label{lem-rankspherical}
    Let $Z\in \mathfrak B(\mathcal X)$, then $\mathrm {rk}(Z)\leq \mathrm{rk}(\mathcal X)$.
\end{lem}

\begin{rmk}\label{rm-changeBcondition}
    In this remark, let the Borel subgroup $B$ of $G$ temporarily be a variable. Recall the following two conditions introduced in Corollary \ref{cor-efficient}: (see Section \ref{sec-chashandDL} for the definition of $\mathrm d_T^B$)
    \begin{itemize}
            \item[$i)_B$] There exists a unique $w\in B(\mathrm k)\backslash G(\mathrm k)/H(\mathrm k)$ such that the reductive quotient of $B\cap w Hw^{-1}$ is a finite algebraic group: that is, the scheme $\mathrm d_T^B(B\cap w Hw^{-1})$ is finite over $\mathrm k$ for every maximal torus $T$ of $B$;
            \item[$ii)_B$] Let $w\in B(\mathrm k)\backslash G(\mathrm k)/H(\mathrm k)$ be the element specified in $i)_B$. Then the group $B\cap w Hw^{-1}$ endowed with the reduced scheme structure is a connected unipotent algebraic group. 
            \end{itemize}
    We easily verify that different choices of  Borel subgroups $B$ of $G$ yield the equivalent conditions $i)_B+ii)_B$. We say that the necessary condition $\star$ holds for $H$ if $i)_B+ii)_B$ holds for some (hence every) Borel pair $(B,T)$ of $G$.
\end{rmk}

\begin{cor}\label{cor-theBorbit}
    Suppose that the condition $i)_B+ii)_B$ listed above holds for some Borel subgroup $B$ of $G$.  Let $w$ be as in Remark \ref{rm-changeBcondition} $i)_B$ and $w^\circ$ be a representative of $w$ in $\mathcal X=G/H$. 
    Then $B\cdot w^\circ$ is the open dense $B$-orbit $\mathcal O_{\mathcal X}^B$ on $\mathcal X$.
\end{cor}
\begin{proof}
    Indeed, we have $\mathrm {rk}(\mathcal X)=\mathrm{rk}(\mathcal O_\mathcal X^B)\leq \mathrm {rk}(B)$. On the other hand, by condition $ii)_B$, we see that $\mathrm {rk}(B\cdot w^\circ)=\mathrm {rk}(B)$. Using Lemma \ref{lem-rankspherical} (let $Z$ be the closure of $B\cdot w^\circ$), we see that $\mathrm {rk}(\mathcal O_\mathcal X^B)=\mathrm {rk}(\mathcal X)=\mathrm {rk}(B\cdot w^\circ)$. Moreover, we see by condition $i)_B$ that $B\cdot w^\circ$ is the unique $B$-orbit on $\mathcal{X}$ with that rank. This forces $\mathcal O_\mathcal X^B=B\cdot w^\circ$.
\end{proof}

\begin{cor}\label{cor-connectedstabilizer}
    Suppose that the necessary condition $\star$ holds for $H$ in the sense of Remark \ref{rm-changeBcondition}. Then for any $x\in \mathcal X$, the corresponding $B$-stabilizer $B_x$ is connected.
\end{cor}
\begin{proof}
    This is a trivial instance of \cite[Corollary 3.3]{Kno}. 
    
    Namely,  Corollary \ref{cor-theBorbit} asserts that the $B$-stabilizer $B_y$ is connected for a general $y\in \mathcal X$. Then we take $X=\mathcal X$ and $w=w_G$ the element of maximal length in \cite[Corollary 3.3]{Kno}.
\end{proof}

\subsection{A partition of $\mathcal Y_{T,B}^H$}
We begin by introducing the notation.

In this paragraph, we fix a Borel pair $(B,T)$ of $G$.
For $w\in  B(\mathrm k)\backslash G(\mathrm k)/H(\mathrm k)$, we set 
$$\mathcal Y_{T,B}^{H,w}:=\{(g,xB,yH)\in \mathcal Y_{T,B}^H:Bx^{-1}yH=BwH\}.$$
Let $i_w:\mathcal Y_{T,B}^{H,w}\to \mathcal Y_{T,B}^H$ denote the natural inclusion.

\begin{lem}\label{lem-dimYTBHw}
    The dimension of $\mathcal Y_{T,B}^{H,w}$ is $\dim (G)$.
\end{lem}
\begin{proof}
    Let $p_{23}:\mathcal Y_{T,B}^{H,w}\to G/B \times G/H$ temporarily denote the projection to the last two factors. Let $w_0\in G$ be a representative of $w$. We see that the image of $p_{23}$ is the $G$-orbit of $(eB,w_0H)\in G/B\times\mathcal X$ with respect to the action $g \cdot (xB,yH)=(gxB,gyH)$. 
 Hence the image of $p_{23}$  is   of dimension $\dim (G)-\dim (B\cap w H w^{-1})$. Moreover, each nonempty fibre of $p_1$ is isomorphic to $B \cap w H w^{-1}$ (via some conjugation). Consequently, we see that $\dim (\mathcal Y_{T,B}^{H,w})=\dim (G)$.
\end{proof}

For convenience, we recall the following well-known lemma. (See \cite[Corollary 2.6]{Kno}.)
\begin{lem}
    The variety $\mathcal X=G/H$ contains only finitely many $B$-orbits. 
\end{lem}

In particular, we see that the set $ B(\mathrm k)\backslash G(\mathrm k)/H(\mathrm k)$ is finite. Combining the above two lemmas, we have the following proposition.

\begin{prop}\label{pro-dimYTBH} 
    The scheme $\mathcal Y_{T,B}^H$ is of dimension $\dim (G)$.
\end{prop}

We have the following easy lemma.

\begin{lem}\label{lem-irreYTBHw}
Suppose that the necessary condition $\star$ holds for $H$ in the sense of Remark \ref{rm-changeBcondition}. 
    Then for any $w\in B(\mathrm k)\backslash G(\mathrm k)/H(\mathrm k)$, the scheme $\mathcal Y_{T,B}^{H,w}$ has a unique irreducible component of dimension $\dim (G)$.
\end{lem}
\begin{proof}
    Let $p_{23}$ be as in the proof of Lemma \ref{lem-dimYTBHw}. Then we see the image is irreducible of dimension $\dim (G)-\dim (B\cap w Hw^{-1})$, and each fibre is likewise irreducible of dimension $\dim (B\cap w Hw^{-1})$ by Corollary \ref{cor-connectedstabilizer}.
\end{proof}

In view of the above lemma, the following definition makes sense.

\begin{defn}\label{def-Yirreduciblecomp}
    Suppose that the necessary condition $\star$ holds for $H$ in the sense of Remark \ref{rm-changeBcondition}. Fix a Borel pair $(B,T)$ of $G$.
    For $w\in B(\mathrm k)\backslash G(\mathrm k)/H(\mathrm k)$, we set $\eta_w$ to be the  generic point of the irreducible component of $\mathcal Y_{T,B}^{H,w}$ of dimension $\dim (G)$.
\end{defn}

\begin{prop}\label{pro-cohomo-directsum}
   Let $\mathscr F$ be a $\Qlb$-sheaf on $\mathcal Y_{T,B}^H$. Then we have a natural isomorphism 
    $$
   \bigoplus_{w\in B(\mathrm k)\backslash G(\mathrm k)/H(\mathrm k)} H^{2\dim (G)}_c(\mathcal Y_{T,B}^{H,w},i_w^*\mathscr F)  \cong H^{2\dim (G)}_c(\mathcal Y_{T,B}^H,\mathscr F).
    $$
\end{prop}
\begin{proof}
    For each $w\in B(\mathrm k)\backslash G(\mathrm k)/H(\mathrm k)$, let $V_w$ temporarily be a $\dim (G)$-dimensional  open subset of $\mathcal Y_{T,B}^{H,w}$ such that it contains  the generic points of all irreducible components of $\mathcal Y_{T,B}^{H,w}$ of dimension $\dim(G)$. We may shrink $V_w$ and assume that $V_w$ is likewise open as a subscheme of $\mathcal Y_{T,B}^H$. Moreover, we may assume that for different $w_1,w_2\in B(\mathrm k)\backslash G(\mathrm k)/H(\mathrm k)$, we have $V_{w_1}\cap V_{w_2}=\emptyset$.
    Let $V$  be the disjoint union of all $V_w$ for $w\in B(\mathrm k)\backslash G(\mathrm k)/H(\mathrm k)$. Let $j_w:V_w\to \mathcal Y_{T,B}^{H,w}$ and $j:V\to \mathcal Y_{T,B}^H$ be the open immersions. The distinguished triangle (induced by the adjunction counit $(j_w)_!\circ j_w^* \to \mathrm {id}$)
    $$
   (j_w)_!\circ j^*_w \circ i_w^* \mathscr F \to i_w^* \mathscr F \to ?
    $$
    yields $H_c^{2\dim (G)}(V_w, j^*_w \circ i_w^* \mathscr F)\cong  H^{2\dim (G)}_c(\mathcal Y_{T,B}^{H,w},i_w^*\mathscr F)$ by  Lemma \ref{lem-dimYTBHw} (and the obvious vanishing of $H^{i}_c(\mathcal Y_{T,B}^H,?)$ for $i\geq 2\dim (G)-1$, since the support of $?$ is of dimension $\leq \dim (G)-1$).
  Similarly,   the distinguished triangle (induced by the adjunction counit $j_!\circ j^*\to \mathrm {id}$)
    $$
    j_!\circ j^* \mathscr F \to  \mathscr F \to ??
    $$
    yields $H_c^{2\dim (G)}(V,j^* \mathscr F)\cong H^{2\dim (G)}_c(\mathcal Y_{T,B}^H,\mathscr F)$. By construction, we have the obvious isomorphism $H^{2\dim (G)}_c(V,j^{*}\mathscr F)=\bigoplus_{w\in B(\mathrm k)\backslash G(\mathrm k)/H(\mathrm k)} H^{2\dim (G)}_c(V_w,j_w^*\circ i_w^*\mathscr F)$. This completes the proof.
\end{proof}

\subsection{Passing to the subscheme of unipotent elements}
Retain the notation as in the previous subsection. In this subsection we will show that the necessary condition $\star$ introduced in Remark \ref{rm-changeBcondition} implies the almost multiplicity-one property.

We need some preliminaries.

\begin{defn} Fix a Borel pair $(B,T)$ of $G$.
  Fix $w\in B(\mathrm k)\backslash G(\mathrm k)/H(\mathrm k)$. We   set $T_w:=\mathrm d_T^B(B\cap w Hw^{-1})$ endowed with the reduced scheme structure, where $\mathrm d_T^B:B\to T$ is the natural map.  Let $r_w:T_w\hookrightarrow T$ be the inclusion.
\end{defn}

Recall the isomorphism $\mathscr R_{T,\chi}\otimes \mathscr I_H\cong (\tau_{T,B}^H)_! \rho_{T,B,H}^* \mathscr L_\chi $ introduced in Remark \ref{rm-badmodel}.
\begin{lem}\label{lem-parition-vanishing}Fix a Borel pair $(B,T)$ of $G$, with $F$-stable $T$.
   Fix $w\in B(\mathrm k)\backslash G(\mathrm k)/H(\mathrm k)$.  Let $\chi:T^F\to \Qlb^\times$ be a character.
    Suppose that the necessary condition $\star$ holds for $H$ in the sense of Remark \ref{rm-changeBcondition}.  Then $H^{k}_c(\mathcal Y_{T,B}^{H,w},i^*_w \rho^*_{T,B,H} \mathscr L_ \chi)=0 $ for all integers $k$,  unless $r_w^* \mathscr L_\chi$ is isomorphic to the constant sheaf $\Qlb$.
\end{lem}

\begin{proof}
    We note that the sheaf $i^*_w \rho^*_{T,B,H} \mathscr L_ \chi$ is equivariant with respect to the action of $G$ on $\mathcal Y_{T,B}^{H,w}$ given by (for $g\in G$ and $(g',xB,yH)\in \mathcal Y_{T,B}^{H,w}$)
    $$g\cdot (g',xB,yH)=(gg'g^{-1},gxB,gyH).$$
    Set $\mathcal O_w$ to be the $G$-orbit of $(eB,w_0H)\in G/B\times G/H$ with respect to the diagonal action, where $w_0\in G$ is a representative of $w$.
     We have a $G$-equivariant map $p_{23}:\mathcal Y_{T,B}^{H,w}\to G/B\times G/H$ given by projection to the last two factors. The image of $p_{23}$ is $\mathcal O_w$.

Let  $\wt w=(eB,w_0H)$. Let $f_{\wt w}=p_{23}^{-1}(\{\wt w\})$ be the fibre of $p_{23}$ along $\wt w$. Let $c_w:f_{\wt w}\hookrightarrow \mathcal Y_{T,B}^{H,w}$ be the closed immersion.
  By the above discussion and proper base change,  it suffices to show:
     \begin{itemize}
         \item[$\heartsuit_1$] Suppose that $r^*_w \mathscr L_\chi$ is not isomorphic to the constant sheaf $\Qlb$ as $\Qlb$-sheaves, then $H^{k}_c(f_{\wt w},c_w^* \circ i^*_w \circ \rho^*_{T,B,H} \mathscr L_ \chi)=0$ for all integers $k$.
     \end{itemize}

We note that the map $\rho_{T,B,H}\circ i_w\circ c_w:f_{\wt w} \to T$ is the restriction of $\mathrm d_T^B$ to $B\cap w_0 H w_0^{-1}$ when we identify $f_{\wt w}$ with $B\cap w_0 H w_0^{-1}$ in an obvious way. Set $B_{w_0}$ to be the closed subscheme of  $B\cap w_0 Hw_0^{-1}$ with the same underlying topological space and the reduced scheme structure.
Let $\mathrm d_T^{B,w_0}:B_{w_0}\to  T_w$ be the  restriction map of $\mathrm d_T^B$. Then $\mathrm d_T^{B,w_0}$ is a surjective homomorphism of algebraic groups, whose kernel is a connected unipotent group by Corollary \ref{cor-connectedstabilizer}.

 In particular, by the projection formula, we see that $H^\bullet_c(f_{\wt w},c_w^* \circ i^*_w \circ \rho^*_{T,B,H} \mathscr L_ \chi)$ can be identified with $H^{\bullet}_c(T_w,r^*_w \mathscr L_ \chi)$ up to shifts and Tate twists.
Consequently, to show $\heartsuit_1$, it suffices to prove:
\begin{itemize}
    \item [$\heartsuit_2$] Suppose that $r^*_w \mathscr L_\chi$ is not isomorphic to $\Qlb$ as $\Qlb$-sheaves, then $H^{k}_c(T_w,r^*_w \mathscr L_ \chi)=0$ for all integers $k$.
\end{itemize}

In what follows, we assume that $r^*_w\mathscr L_\chi$ is not isomorphic to the constant sheaf. 
We want to show $H^{k}_c(T_w,r_w^* \mathscr L_ \chi)=0$ for all integers $k$.
We note by Corollary \ref{cor-connectedstabilizer} that $T_w$ is connected. Hence $T_w$ is a torus.

We may choose a sufficiently divisible integer $n$ such that $T_w$ is $F^n$-stable. Set $\chi_w: T_w^{F^n} \to \Qlb$ to be the restriction of $\chi\circ\mathrm N_{T}^n:T^{F^n}\to\Qlb^\times$ to $T^{F^n}_w$. 
We have an isomorphism $r^*_w\mathscr L_\chi^T \cong \mathscr L_{\chi_w}^{T_w}$ (of sheaves endowed with  $(F^n)^*$ actions) in view of the explicit description of the corresponding functions Remark \ref{rm-torishf} and Lemma \ref{lem-chebo-tori} (applied to the torus $T_w$ defined over $\mathbb F_{q^n}$). 
(Here, the sheaf $\mathscr L_{\chi_w}^{T_w}$ is obtained in Remark \ref{rm-gentorshf} for the datum $n,T_w,\chi_w$.)
Then $\chi_w$ is nontrivial since $r^*_w\mathscr L_\chi^T\cong \mathscr L_{\chi_w}^{T_w}$ is not isomorphic to the constant sheaf. Hence we have $H^{k}_c(T_w,r_w^* \mathscr L_ \chi)=0$ for all integers $k$ by Lemma \ref{lem-vanis-tori}. This completes the proof.

Alternatively, we may use the map introduced in Remark \ref{rm-kummersheaf} and the vanishing  introduced in Remark \ref{rm-vanishin-tori} to see $\heartsuit_2$. 
\end{proof}

\begin{defn}\label{def-not-open-orbit}
    For a Borel pair $(B,T)$ of $G$, we set $$\Theta_{T,B}:=\{w\in B(\mathrm k)\backslash G(\mathrm k)/H(\mathrm k):\text{$B\cdot wH $ is \textbf{not} the open dense $B$-orbit of $\mathcal X=G/H$}\}.$$
\end{defn}

\begin{prop}\label{pro-vanishing-estimation}
    Suppose that the necessary condition $\star$ holds for $H$ in the sense of Remark \ref{rm-changeBcondition}, then there exists a positive integer  $N_{G,H}$ (which is constructed explicitly in the proof) satisfying the following:
    \begin{itemize}
        \item 
 Choose an arbitrary positive integer $n$, and choose any Borel pair $(B,T)$ of G with $F^n$-stable $T$. 
    Then we have
\begin{align*}
    &\# \{\chi \in (T^{F^n})^\vee: \text{$r_w^* \mathscr L_{\chi}^T$ is not isomorphic to the constant sheaf $\Qlb$ for all $w\in \Theta_{T,B}$. }\}\\
    \geq & (1-N_{G,H}\frac{(q^n+1)^{\dim (T)-1}}{(q^n-1)^{\dim(T)}})\# ((T^{F^n})^\vee).
\end{align*}
    \end{itemize}
    Here, for $n,T,B,\chi$ as listed above, the sheaf $\mathscr L_\chi^T$ is obtained in Remark \ref{rm-gentorshf}.  
\end{prop}

\begin{proof}

In this paragraph, we fix a positive integer $k$ and a Borel pair $(B,T)$ with $F^k$-stable $T$.
Let $\mathrm t_w:T^*(\mathrm k)\to T_w^*(\mathrm k)$ be the transpose map of $r_w$ for all $w\in B(\mathrm k)\backslash G(\mathrm k)/H(\mathrm k)$. 
 Set $X_{T,B,k}$ to be the minimal reduced subscheme of $T^*$ that is stable under the endomorphism $F^k$ and contains $\ker{(\mathrm t_w)}$ for any 
   $w\in \Theta_{T,B}$. Note that $X_{T,B,k}$ is indeed a finite union of subtori (of $T^*$) of dimension $\leq \dim (T)-1$ by the assumption introduced in Remark \ref{rm-changeBcondition} and the connectedness introduced in  Corollary \ref{cor-connectedstabilizer}. Further, the scheme $X_{T,B,k}$ is defined over $\mathbb F_{q^k}$. 

In the remainder of this proof, we fix an $F$-stable  Borel pair $(\mathrm B,\mathrm T)$ of $G$.
Let $\mathrm t^\circ_v:\mathrm T^*(\mathrm k)\to \mathrm T_v^*(\mathrm k)$ be the transpose map of $r_v$ for any $v\in \mathrm B(\mathrm k)\backslash G(\mathrm k)/H(\mathrm k)$.    Set $X_{\mathrm T,\mathrm B}$ to be the minimal reduced closed subscheme of $T^*$ that is stable under:
\begin{itemize}
    \item the endomorphism $F$, and
    \item the action of the Weyl group $W(\mathrm T)$;
\end{itemize}
  and that contains $\ker{(\mathrm t_v^\circ)}$ for every
   $v\in \Theta_{\mathrm T,\mathrm B}$. Note that $X_{\mathrm T,\mathrm B}$ is indeed a finite union of subtori (of $\mathrm T^*$) of dimension $\leq \dim (\mathrm T)-1$ by the assumption introduced in Remark \ref{rm-changeBcondition} and the connectedness introduced in  Corollary \ref{cor-connectedstabilizer}. 

 In this paragraph, let $k$ and $(B,T)$  be as in the first paragraph of this proof. We may choose $g_{T,B}\in G(\mathrm k)$ such that $g_{T,B}^{-1} \mathrm T g_{T,B}=T$ ($g_{T,B}^{-1} \mathrm B g_{T,B}=B$, resp.). 
 Set $c_{T,B}: T\to \mathrm T$ to be the isomorphism defined by $t\mapsto g_{T,B}tg_{T,B}^{-1} $. Let $c^t_{T,B}:\mathrm T^*\to  T^*$ be the transpose of $c_{T,B}$. In view of the explicit relation between  Frobenius structures on $T$ and $\mathrm T$ (see \cite[Corollary 1.14]{DL} and what follows), we see that $X_{T,B,k}$ is a subscheme of $c^t_{T,B} (X_{\mathrm T,\mathrm B})$, and that the scheme $c^t_{T,B} (X_{\mathrm T,\mathrm B})$ is $F^k$-stable. Consequently, by \cite[Lemma 8.2]{L5}, we have 
  $$
   \# ((T^*)^{F^k}-X^{F^k}_{T,B,k})\geq \# ((T^*)^{F^k}-(c^t_{T,B}(X_{\mathrm T,\mathrm B}))^{F^k}) \geq 
   (1-\mathbf N(\mathrm T^*,X_{T,B})\frac{(q^k+1)^{\dim (T)-1}}{(q^k-1)^{\dim(T)}})\# ((T^*)^{F^k}).
   $$
(See Remark \ref{rm-numberNTX} for the definition of the number $\mathbf N(\_,\_)$, notice that $\mathbf N(\mathrm T^*,X_{\mathrm T,\mathrm B})=\mathbf N(T^*,c^t_{T,B}(X_{\mathrm T,\mathrm B})) $ by definition.)
    In view of the commutative diagram introduced in Remark  \ref{rm-kummersheaf}, the above inequality indicates
   \begin{align*}
       &\# \{\chi \in (T^{F^k})^\vee: \text{$r_w^* \mathscr L_{\chi}^T$ is not isomorphic to the constant sheave $\Qlb$ for $w\in \Theta_{T,B}$. }\} \\&\geq (1-\mathbf{N}(\mathrm T^*,X_{\mathrm T,\mathrm B})\frac{(q^k+1)^{\dim (T)-1}}{(q^k-1)^{\dim(T)}})\# ((T^*)^{F^k}).
   \end{align*}

We set $N_{G,H}:=\mathbf N(\mathrm T^*,X_{\mathrm T,\mathrm B})$ and complete the proof.
\end{proof}

We can now state the main result of this subsection.

\begin{thm}\label{thm-almostmulone}
    Suppose that the necessary condition $\star$ holds for $H$ in the sense of Remark \ref{rm-changeBcondition}. Then the almost muptiplicity-one property holds for $H$: that is, we have
    $$
    \lim_{n\to \infty} \frac{\#\mathcal E^{\diamondsuit}(G,H,n)}{\#\mathcal E(G,n)}=1.
    $$
    (See Definition \ref{def-almostmulone}.)
\end{thm}
\begin{proof}
In this paragraph, fix an $F$-stable maximal torus $T$ of $G$ and a Borel subgroup $B$ of $G$ containing $T$. In this paragraph, we assume that $q\geq N_G$ (see Remark \ref{rm-qsize} for the integer $N_G$). Let $N_{G,H}$ be the integer introduced in Proposition \ref{pro-vanishing-estimation}.
By Proposition \ref{pro-vanishing-estimation}, Lemma \ref{lem-parition-vanishing} and Remark \ref{rm-badmodel}, we see that there exist at least $(1-N_{G,H}\frac{(q+1)^{\dim (T)-1}}{(q-1)^{\dim(T)}})\# (T^{F})$ many characters $\chi$ of $T^{F}$ such that  Proposition \ref{pro-cohomo-directsum}  exhibits an isomorphism
$$H_c^{2\dim (G)}(\mathcal Y_{T,B}^{H,w_0},i_{w_0}^* \rho_{T,B,H}^*\mathscr L_\chi)\cong H^{2\dim G}_c(G,\mathscr R_{T,\chi}\otimes \mathscr I_H),$$ where $w_0$ is the unique elements of $B(\mathrm k)\backslash G(\mathrm k)/H(\mathrm k)$ such that $B\cdot w_0H$ is the open $B$-orbit on $\mathcal X=G/H$. We note that $i_{w_0}^* \rho_{T,B,H}^*\mathscr L_\chi$ is isomorphic to the $1$-rank constant sheave on $\mathcal Y_{T,B}^{H,w_0}$ by definition. 
Hence the dimension of $H_c^{2\dim (G)}(\mathcal Y_{T,B}^{H,w_0},i_{w_0}^* \rho_{T,B,H}^*\mathscr L_\chi)$ is $1$ given Lemma \ref{lem-irreYTBHw}.
Consequently, for such $\chi$ as  above, we see that $\mathrm P _{T,\chi}(1)$ is a root of unity by Theorem  \ref{thm-pertopcohomo} and Corollary \ref{cor-cohomoweight}. If we further assume the above $\chi$ is regular, then by definition $(-1)^{\sigma(T)+\sigma(G)}\mathrm P_{T,\chi}(1)$ is a positive  integer (see Proposition \ref{regulardl-irr}), yielding $\mathrm P_{T,\chi}(1)=(-1)^{\sigma(T)+\sigma(G)}$. 

In this paragraph, we assume the variable integer $n$ satisfies $q^n\geq N_G$.
Similarly to the first paragraph of this proof, we see that there exist at least
$(1-N_{G,H}\frac{(q^n+1)^{\dim (T)-1}}{(q^n-1)^{\dim(T)}})\# (T^{F^n})$ many characters $\chi$ of $T^{F^n}$ such that $\langle R_{T,\chi}^{(n)},1_{H^{F^n}}\rangle_{{H^{F^n}}}$ is a root of unity for every $F^n$-stable maximal torus $T$ of $G$.
(In proving this claim, all ingredients in the first paragraph should be replaced by their corresponding versions over $\mathbb F_{q^n}$. Note that the integer $N_{G,H}$ remains constant as $n$ varies, by  Proposition \ref{pro-vanishing-estimation}.)
Consequently, we see that 
$$1=\lim\limits_{m\to \infty}
\frac{\#\{\text {regular Deligne-Lusztig characters $\pi$ of $G^{F^m}$ such that $\langle \pi, 1_{H^{F^m}}\rangle_{H^{F^m}}=\pm 1$ }\}}{\#\{\text {regular Deligne-Lusztig characters  of $G^{F^m}$ }\}}.
$$
(Here, we also use the estimation (\ref{ineq-regularchara}) for the size of the set of regular characters.)

    By Proposition \ref{pro-dl-vs-irrerep}, Proposition \ref{regulardl-irr} and the discussion in the last paragraph, we see that
     $$
    \lim_{n\to \infty} \frac{\#\mathcal E^{\diamondsuit}(G,H,n)}{\#\mathcal E(G,n)}=1,
    $$
    as desired.
\end{proof}

Set $\mathcal U$ to be the closed subscheme of $G$ consisting of unipotent elements. Set $i_{\mathcal U}:\mathcal U\hookrightarrow G$ to be the natural inclusion. 
\begin{defn}Fix a Borel pair $(B,T)$ of $G$.
    Recall the morphism $\tau_{T,B}^H : \mathcal Y_{T,B}^H\to G$ introduced in Definition \ref{def-scheme-badmodel}. Let $\tau_{T,B}^{H,\mathcal U}:\mathcal Y_{T,B}^{H,\mathcal U}\to \mathcal U $ be the pullback of $\tau_{T,B}^H$ along $i_\mathcal U$. 
    Set $i_{\mathcal Y,\mathcal U}:\mathcal Y_{T,B}^{H,\mathcal U}\hookrightarrow \mathcal Y_{T,B}^H$ to be the natural inclusion. Let $U^c_{T,B}$  denote the open complement of $\mathcal Y_{T,B}^{H,\mathcal U}$ in $\mathcal Y_{T,B}^H$, and let $j_\mathcal U:U^c_{T,B}\hookrightarrow \mathcal Y_{T,B}^H$ be the open immersion. 
\end{defn}

For the cohomology space, we have the following tautological corollary.
\begin{cor}\label{cor-adjunction-unipotent}
Fix  an $F$-stable maximal torus $T$ of $G$. 
   Suppose that $q\geq N_G$, where $N_G$ is the integer introduced in Remark \ref{rm-qsize}. 
   Suppose further that the necessary condition $\star$ holds for $H$ in the sense of Remark \ref{rm-changeBcondition}. Then for at least $(1-N_{G,H}\frac{(q+1)^{\dim (T)-1}}{(q-1)^{\dim(T)}})\# (T^{F})$ many characters $\chi$ of  $T^F$, the adjunction induces isomorphisms (here, $N_{G,H}$ is the integer introduced in Proposition \ref{pro-vanishing-estimation})
   $$
   H^{2\dim G}_c(G,\mathscr R_{T,\chi}\otimes \mathscr I_H)\xrightarrow{\sim} H^{2\dim G}_c\left(G,(i_\mathcal U)_*\circ i_{\mathcal U}^* (\mathscr R_{T,\chi}\otimes \mathscr I_H)\right).
   $$
\end{cor}

\begin{proof}Let $B$ be a (not necessarily $F$-stable) Borel group of $G$ containing $T$.
    
    By Remark \ref{rm-badmodel} and proper base change, it is equivalent to deal with the map induced by adjunction along $i_{\mathcal Y,\mathcal U}$
    \begin{equation}\label{eq-adjunipotent}
          H_c^{2\dim (G)}(\mathcal Y_{T,B}^H, \rho_{T,B,H}^* \mathscr L_\chi ) \to  H_c^{2\dim (G)}(\mathcal Y_{T,B}^{H,\mathcal U},i_{\mathcal Y,\mathcal U}^* \circ\rho_{T,B,H}^* \mathscr L_\chi ).
    \end{equation}

 By Proposition \ref{pro-vanishing-estimation}, it suffices to show:
    \begin{itemize}
        \item For a character $\chi:T^F\to \Qlb^\times$ such that $r_w^* \mathscr L_{\chi}^T$ is not isomorphic to the constant sheaf $\Qlb$ for every $w\in \Theta_{T,B}$ (as in Proposition \ref{pro-vanishing-estimation}), the morphism (\ref{eq-adjunipotent}) is an isomorphism.
    \end{itemize}
In what follows, we fix one such  $\chi$. Let $\mathscr F$ temporarily denote $\rho^*_{T,B,H}\mathscr L_\chi$ for simplicity.

      In view of the distinguished triangle
    $$
    (j_\mathcal U)_!j_\mathcal U^* \mathscr F\to \mathscr F \to  (i_{\mathcal Y,\mathcal U})_*i_{\mathcal Y,\mathcal U}^*\mathscr F
    $$
    we reduce the problem to show $H^{2\dim (G)}_c\left(\mathcal Y_{T,B}^H,(j_\mathcal U)_!j_\mathcal U^* \mathscr F\right)=0.$

    We claim that \begin{itemize}
        \item there exists an isomorphism $\bigoplus\limits_{w\in \Theta_{T,B}}H^{2\dim (G)}_c(\mathcal Y_{T,B}^{H,w},i_w^*\mathscr F)\cong H^{2\dim (G)}_c\left(\mathcal Y_{T,B}^H,(j_\mathcal U)_!j_\mathcal U^* \mathscr F\right).$
    \end{itemize}
   Granting this claim, we then see $H^{2\dim (G)}_c\left(\mathcal Y_{T,B}^H,(j_\mathcal U)_!j_\mathcal U^* \mathscr F\right)=0$ by Lemma \ref{lem-parition-vanishing} and our choice of $\chi$.

   It remains to show the above claim. Let $w_0\in B(\mathrm k)\backslash G(\mathrm k)/H(\mathrm k)$ be the element such that $B\cdot w_0H$ is the open $B$-orbit in $\mathcal X=G/H$.
   Let $j:V\hookrightarrow \mathcal Y_{T,B}^H$ be as in the proof of Proposition \ref{pro-cohomo-directsum}. By Lemma \ref{lem-irreYTBHw} and Lemma \ref{lem-YwintersectU} below, the open subscheme $U^c_{T,B}$ contains $\eta_w$ (see Definition \ref{def-Yirreduciblecomp}) for all $w\in \Theta_{T,B}$ and does not contain $\eta_{w_0}$.
   Hence we may shrink $V$ such that $V$ is the disjoint union of $V\cap U_{T,B}^c$ and $V_{w_0}$, where $V_{w_0}$ is an open subscheme of $\mathcal Y_{T,B}^{H,w_0}$ of dimension $\dim (G)$. Using distinguished triangles arising from appropriate excisions, we see that there are natural isomorphisms  
   $$
  \bigoplus\limits_{w\in \Theta_{T,B}}H^{2\dim (G)}_c(\mathcal Y_{T,B}^{H,w},i_w^*\mathscr F)\xleftarrow{\sim} H^{2\dim (G)}_c(V\cap U_{T,B}^c,\mathscr F_{|V\cap U_{T,B}^c})\xrightarrow{\sim}H^{2\dim (G)}_c(U_{T,B}^c,\mathscr F_{|U_{T,B}^c}).
   $$
   Moreover, we have a canonical identification $H^{2\dim (G)}_c\left(\mathcal Y_{T,B}^H,(j_\mathcal U)_!j_\mathcal U^* \mathscr F\right)\cong H^{2\dim (G)}_c(U_{T,B}^c,\mathscr F_{|U_{T,B}^c})$,
   as desired.
\end{proof}

In the proof of the above corollary, we need the following lemma.

\begin{lem}\label{lem-YwintersectU}
    Fix a Borel pair $(B,T)$ of $G$.  Suppose that the necessary condition $\star$ holds for $H$ in the sense of Remark \ref{rm-changeBcondition}. Then, for $w\in \Theta_{T,B}$, the closed subscheme $\mathcal Y_{T,B}^{H,\mathcal U}\cap \mathcal Y_{T,B}^{H,w}$ of $\mathcal Y_{T,B}^{H,w}$ is of dimension $\leq \dim (G)-1$. Moreover, for $w_0\in B(\mathrm k)\backslash G(\mathrm k)/H(\mathrm k)$ the specific element such that $B\cdot w_0H$ is the open $B$-orbit on $\mathcal X=G/H$, the closed subscheme $\mathcal Y_{T,B}^{H,\mathcal U}\cap \mathcal Y_{T,B}^{H,w_0}$ of $\mathcal Y_{T,B}^{H,w_0}$ is of dimension $\dim (G)$.
\end{lem}

\begin{proof}
  By our assumption,  for all $(g,xB,yH)\in\mathcal Y_{T,B}^{H,w_0}(\mathrm k)$, the element $x^{-1}gx$ is unipotent. This shows that the closed subscheme $\mathcal Y_{T,B}^{H,\mathcal U}\cap \mathcal Y_{T,B}^{H,w_0}$ of $\mathcal Y_{T,B}^{H,w_0}$ contains all $\mathrm k$-points of $\mathcal Y_{T,B}^{H,w_0}$. We win in this case by using Lemma \ref{lem-irreYTBHw}. 

   Fix $w\in \Theta_{T,B}$ in this paragraph. 
   We may rephrase $$\mathcal Y_{T,B}^{H,w}\cap \mathcal Y_{T,B}^{H,\mathcal U}=\{(g,xB,yH)\in \mathcal Y_{T,B}^{H,w}:\mathrm d_T^B(x^{-1}gx)=e\},$$ where $e$ denotes the identity element. By our assumption and Corollary \ref{cor-connectedstabilizer}, the scheme $$\mathrm{d}_T^B(B\cap wHw^{-1})$$ has positive dimension, where we denote a representative of $w$ in $G(\mathrm k)$ again by $w$. 
   Let $$p_{23}:\mathcal Y_{T,B}^{H,w}\cap \mathcal Y_{T,B}^{H,\mathcal U}\to G/B\times G/H$$ temporarily be the projection to the last two factors. Then
we see that $p_{23}$ has its image of dimension $\dim (G)-\dim (B\cap wHw^{-1})$, and each nonempty fibre of $p_{23}$ is of dimension $$\dim (B\cap w Hw^{-1})-\dim (\mathrm{d}_T^B(B\cap wHw^{-1})).$$ Consequently, we have $\dim(\mathcal Y_{T,B}^{H,w}\cap \mathcal Y_{T,B}^{H,\mathcal U})=\dim (G)-\dim (\mathrm{d}_T^B(B\cap wHw^{-1}))$. This completes the proof.
   
\end{proof}

We note that the inclusion $i_\mathcal U:\mathcal U\hookrightarrow G$ is defined over $\mathbb F_q$. Consequently, the derived object $(i_\mathcal U)_*\circ i_{\mathcal U}^* (\mathscr R_{T,\chi}\otimes \mathscr I_H)$ has been endowed with a natural Weil structure, which is induced by the Weil structure of $\mathscr R_{T,\chi}\otimes \mathscr I_H$.

\begin{cor}\label{cor-unipotent-top-cohomo}
    Let the notation and the assumption be as in Corollary \ref{cor-adjunction-unipotent}. Let $\chi:T^F\to \Qlb^\times$ be a character as  in Corollary \ref{cor-adjunction-unipotent}. Suppose further that $\chi $ is regular. Then $H^{2\dim G}_c\left(G,(i_\mathcal U)_*\circ i_{\mathcal U}^* (\mathscr R_{T,\chi}\otimes \mathscr I_H)\right)$ is of dimension $1$. Moreover, we have
    $$
    \mathrm{Tr}(F^*,H^{2\dim G}_c\left(G,(i_\mathcal U)_*\circ i_{\mathcal U}^* (\mathscr R_{T,\chi}\otimes \mathscr I_H)\right))=(-1)^{\sigma (T)+\sigma(G)}q^{\dim (G)}.
    $$
\end{cor}
\begin{proof}For the dimension of $H^{2\dim G}_c\left(G,(i_\mathcal U)_*\circ i_{\mathcal U}^* (\mathscr R_{T,\chi}\otimes \mathscr I_H)\right)$, combine  Lemma \ref{lem-irreYTBHw},  \ref{lem-YwintersectU} and Remark \ref{rm-badmodel}. (The proof is similar to that of Corollary \ref{cor-adjunction-unipotent}.)
  It remains to show the last identity.
    By Theorem \ref{thm-pertopcohomo} and Corollary \ref{cor-adjunction-unipotent}, we have
    $$
  \langle R_{T,\chi},1_{H^F}\rangle_{H^F}= \mathrm P_{T,\chi}(1)= \frac{1}{q^{\dim (G)}}\mathrm{Tr}(F^*,H^{2\dim G}_c\left(G,(i_\mathcal U)_*\circ i_{\mathcal U}^* (\mathscr R_{T,\chi}\otimes \mathscr I_H)\right))
    $$
    Since $\chi$ is assumed to be regular, we see that $(-1)^{\sigma(T)+\sigma(G)} R_{T,\chi}$ is the character of an irreducible representation of $G^F$ by Proposition \ref{regulardl-irr}. Hence the left-hand side of the above equation is of the form $(-1)^{\sigma(T)+\sigma(G)}k$ for some positive integer $k$.  The right-hand side of the above equation is a root of unity by Corollary \ref{cor-cohomoweight} and Corollary \ref{cor-adjunction-unipotent}. And we conclude the proof.
\end{proof}

\begin{rmk}
    It is well known that the isomorphism class of (the complexes of) Weil sheaves $i^{*}_\mathcal U\mathscr R_{T,\chi}$ is independent of the choice of the character $\chi$. Consequently, the isomorphism class of the complex of Weil sheaves $(i_\mathcal U)_*\circ i_{\mathcal U}^* (\mathscr R_{T,\chi}\otimes \mathscr I_H)$ is likewise independent of the choice of $\chi$. But we systematically ignore this fact in the rest of this paper.
\end{rmk}

\subsection{Refinement}We retain the notation as in the previous subsection.

 This subsection is dedicated to proving the following theorem.
 \begin{thm}\label{thm-refine-main}
     Suppose that the necessary condition $\star$ holds for $H$ in the sense of Remark \ref{rm-changeBcondition}. Then for any Borel subgroup $B$ of $G$, the group $B\cap w H w^{-1}$ (endowed with the reduced scheme structure) introduced in $ii)_B$ of Remark \ref{rm-changeBcondition} is the trivial algebraic group. 
 \end{thm}

We need some preliminaries.

\begin{lem}\label{lem-unipotent-weight}
    In the situation of Corollary \ref{cor-unipotent-top-cohomo}, we have 
    $$H^{i}_c\left(G,(i_\mathcal U)_*\circ i_{\mathcal U}^* (\mathscr R_{T,\chi}\otimes \mathscr I_H)\right)=0
    $$
    for integers $i\geq 2\dim (G)+1$ or $i\leq-1$.
    For $0\leq j\leq 2\dim (G)-1$ and eigenvalues $\alpha$ of $F^*$ on $H^{j}_c\left(G,(i_\mathcal U)_*\circ i_{\mathcal U}^* (\mathscr R_{T,\chi}\otimes \mathscr I_H)\right),
    $
    we have $$
    |\alpha|\leq q^{j/2}
    $$
    for all identifications $\Qlb\cong \mathbb C$.
\end{lem}
\begin{proof}
    The proof is essentially identical to that of Corollary \ref{cor-cohomoweight}, which we sketch below.

Fix a Borel subgroup $B$ of $G$ containing $T$.
    We replace the field $\mathbb F_q$ by an extension of sufficiently divisible degree $N$ such that $B$ is $F^N$-stable. Then we identify $H^{i}_c\left(G,(i_\mathcal U)_*\circ i_{\mathcal U}^* (\mathscr R_{T,\chi}\otimes \mathscr I_H)\right)$ with $H^{i}_c(\mathcal Y_{T,B}^{H,\mathcal U},i^*_{\mathcal Y,\mathcal U} \circ\rho^*_{T,B,H}\mathscr L_\chi)$. This identification is compatible with the action of $(F^{N})^*$ by Corollary \ref{cor-FstableBorel}. Our assertion follows from the fact that $\dim \mathcal Y_{T,B}^{H,\mathcal U}\leq \dim (G)$ (Proposition \ref{pro-dimYTBH}) and Deligne's weight theory \cite[Th\'eor\`eme 3.3.1]{D}.
\end{proof}
We need some facts about Steinberg representations. Let $\mathrm {St}_G$ denote the character of the Steinberg representation of $G^{F}$.

The following lemma  follows from  \cite[(7.14.2)]{DL}.
\begin{lem}\label{lem-steinberg-combi}
    
    We have
    $$
    \mathrm{St}_G=\sum\limits_{(T)}\frac{(-1)^{\sigma(T)+\sigma(G)}}{\#W(T)^F} R_{T,1},
    $$
    where $\sum\limits_{(T)}$ means sum over the $G^F$-conjugacy classes of $F$-stable maximal tori of $G$.
\end{lem}

The following lemma follows from \cite[Theorem 7.1]{DL}.
\begin{lem}\label{lem-steinberg-value}
    For any $F$-stable maximal torus $T$ of $G$, we have
    $$
    R_{T,1}(e)=(-1)^{\sigma (T)+\sigma(G)}\frac{\#G^F}{\#T^F \cdot\mathrm {St}_G(e)}.
    $$
Moreover, for nontrivial unipotent $u\in G^F$, we have $\mathrm {St}_G(u)=0$.
\end{lem}

We need the following elementary fact. (See \cite[Theorem 4.2]{DL}.)

\begin{lem}\label{lem-green-vs-dl}
    Fix any $F$-stable maximal torus $T$ and character $\chi:T^F\to \Qlb^\times$. We have
    $$
    R_{T,\chi}(u)=R_{T,1}(u)
    $$
    for any unipotent $u\in G^F$. 
\end{lem}

\begin{proof}[Proof of Theorem \ref{thm-refine-main}]Note that the assertion is geometrical.
    By replacing $\mathbb F_q$ with an extension of sufficiently large degree, we may assume:\begin{itemize}
        \item[i)] For any $F$-stable maximal torus $T$ of $G$, there exists a regular  character $\chi:T^F\to \Qlb^\times $ as in Corollary \ref{cor-unipotent-top-cohomo}; 
        \item[ii)] For any $F$-stable maximal torus $T$ of $G$ and character $\chi:T^F\to \Qlb^\times$, Theorem \ref{thm-pertopcohomo} holds for  integers $n\geq 1$.
    \end{itemize}
  (And we do assume so in the remainder of this proof.)  

    Fix a rational Borel pair $(\mathrm B_0,\mathrm T_0)$ of $G_0$ over $\mathbb F_q$. Let $\mathrm T\subset \mathrm B$ denote the pullback of $\mathrm T_0\subset \mathrm B_0$ to $\mathrm k$. Then $(\mathrm B,\mathrm T)$ is an $F$-stable Borel pair of $G$.
For each $F$-stable maximal torus $T$ of $G$, we choose a regular character $\chi_T:T^F\to \Qlb^\times$ as in i) above.
We choose and fix a positive integer $k$ satisfying the following: ( Set $\mathrm W=W(\mathrm T)$ to be the Weyl group of $\mathrm T$. )
\begin{itemize}
    \item[$\heartsuit_1$] $F^k$ acts trivially on the Weyl group $\mathrm W$;
    \item[$\heartsuit_2$] $\mathrm T_0\otimes_{\mathbb F_q} \mathbb F_{q^k}$ is split as a torus over $\mathbb F_{q^k}$.
\end{itemize}
Let $K=k|\mathrm W|$.
Choose a set $\mathcal T$ of  representatives of the $G^F$-conjugacy classes of $F$-stable maximal tori of $G$. 
In view of \cite[Proposition 1.16]{DL}, we see from $\heartsuit_1$ and $\heartsuit_2$ that \begin{itemize}
    \item For any positive integer $i$, the set $\mathcal T$ represents the $G^{F^{1+Ki}}$-conjugacy classes of $F^{1+Ki}$-stable maximal tori of $G$. 
    \item As $i$ ranges over  positive integers, the values $(-1)^{\sigma_{1+Ki}(G)}$ and $(-1)^{\sigma_{1+Ki}(T)}$ remain constant for $T\in \mathcal T$, where $\sigma_{1+Ki}(G)$ (resp., $\sigma_{1+Ki}(T)$) denotes the $F^{1+Ki}$-rank of $G$ (resp., $T$).
\end{itemize}

By Theorem \ref{thm-pertopcohomo}, Corollary \ref{cor-adjunction-unipotent} and Corollary \ref{cor-unipotent-top-cohomo}, we have
$$
1=\sum_{T\in \mathcal T}\frac{(-1)^{\sigma(T)+\sigma(G)}}{\#W(T)^F}\mathrm P_{T,\chi_T}(1)=\sum_{T\in \mathcal T}\frac{(-1)^{\sigma(T)+\sigma(G)}}{\#W(T)^F\cdot q^{\dim (G)}}\mathrm{Tr}(F^*,H^{2\dim G}_c\left(G,(i_\mathcal U)_*\circ i_{\mathcal U}^* (\mathscr R_{T,\chi_T}\otimes \mathscr I_H)\right)).
$$
By Lemma \ref{lem-unipotent-weight} and Corollary \ref{cor-unipotent-top-cohomo}, we have
\begin{align*}
    &\lim\limits_{i\to \infty} \frac{1}{q^{(1+Ki)\dim (G)}}\mathrm{Tr}((F^{1+Ki})^*,H^{\bullet}_c\left(G,(i_\mathcal U)_*\circ i_{\mathcal U}^* (\mathscr R_{T,\chi_T}\otimes \mathscr I_H)\right))\\=&\frac{1}{q^{\dim(G)}}\mathrm{Tr}(F^*,H^{2\dim G}_c\left(G,(i_\mathcal U)_*\circ i_{\mathcal U}^* (\mathscr R_{T,\chi_T}\otimes \mathscr I_H)\right)).
\end{align*}
Hence we have
$$
1= \lim\limits_{i\to \infty}\frac{1}{q^{(1+Ki)\dim (G)}}\sum_{T\in \mathcal T}\frac{(-1)^{\sigma(T)+\sigma(G)}}{\#W(T)^F} \mathrm{Tr}((F^{1+Ki})^*,H^{\bullet}_c\left(G,(i_\mathcal U)_*\circ i_{\mathcal U}^* (\mathscr R_{T,\chi_T}\otimes \mathscr I_H)\right)).
$$
It suffices to show the conclusion for $B=\mathrm B$ by the obvious symmetry. Note that for every positive integer $n$, the function $K^n_{\mathscr R_{T,\chi_T}\otimes \mathscr I_H}:G^{F^n}\to \Qlb$ corresponding to the complex $\mathscr R_{T,\chi_T}\otimes \mathscr I_H$ is the virtual character $R_{T,\chi_T \circ \mathrm N_T^n}^{(n)}\cdot \mathrm{Ind}_{H^{F^n}}^{G^{F^n}}(1_{H^{F^n}})$. (For the notation, see Proposition \ref{prop-charactersheaf-main} and the paragraph below Definition \ref{def-periods}.)
By Grothendieck trace formula, Lemma \ref{lem-steinberg-combi}, Lemma \ref{lem-green-vs-dl}, Lemma \ref{lem-steinberg-value} and our assumption on $K$, we have
\begin{align*}
    &\sum_{T\in \mathcal T}\frac{(-1)^{\sigma(T)+\sigma(G)}}{|W(T)^F|} \mathrm{Tr}((F^{1+Ki})^*,H^{\bullet}_c\left(G,(i_\mathcal U)_*\circ i_{\mathcal U}^* (\mathscr R_{T,\chi_T}\otimes \mathscr I_H)\right))\\
    =&\sum_{u\in \mathcal U^{F^{1+Ki}}} \mathrm {St}_G^{1+Ki}(u) \mathrm {Ind}_{H^{F^{1+Ki}}}^{G^{F^{1+Ki}}}(1_{H^{F^{1+Ki}}})(u)\\
    =&\frac{\#G^{F^{1+Ki}}\cdot\#\mathcal X^{F^{1+Ki}}}{\#\mathrm T^{F^{1+Ki}}\cdot\#(G/\mathrm B)^{F^{1+Ki}}},
\end{align*}
where $\mathrm {St}_G^{1+Ki}$ denotes the character of the Steinberg representation of $G^{F^{1+Ki}}$. (Here we implicitly use the fact $R_{\mathrm T,1}^{(1+Ki)}(e)=|(G/\mathrm B)^{F^{1+Ki}}|$, see the discussion before Definition \ref{de-regratio} for the notation.) Consequently, we have (using Lemma \ref{lem-j-dimscheme} below)
$$
1=\lim\limits_{i\to \infty}\frac{\#G^{F^{1+Ki}}\cdot\#\mathcal X^{F^{1+Ki}}}{q^{(1+Ki)\dim(G)}\cdot \#\mathrm {T}^{F^{1+Ki}}\cdot\#(G/\mathrm B)^{F^{1+Ki}}}=\lim\limits_{i\to \infty}\frac{\#\mathcal {X}^{F^{1+Ki}}}{\#\mathrm T^{F^{1+Ki}}\cdot\#(G/\mathrm B)^{F^{1+Ki}}}=\lim\limits_{i\to \infty}\frac{\#\mathcal {X}^{F^{1+Ki}}}{\#\mathrm {B}^{F^{1+Ki}}}.
$$
This indicates $\dim (\mathcal X)=\dim (\mathrm B)$ by Lemma \ref{lem-j-dimscheme}. Thus, the open dense $\mathrm B$-orbit $\mathcal O_\mathcal X^\mathrm B$ has dimension $\dim (\mathrm B)$. Then Corollary \ref{cor-theBorbit} and  \ref{cor-connectedstabilizer} show that $\mathrm B_w^{red}$ is the trivial algebraic group, where $\mathrm B_w$ denotes the $\mathrm B$-stabilizer of $w_0 \in \mathcal X$ (a representative of $w$), and $\mathrm B_w^{red}$ is its reduced subgroup scheme.
\end{proof}

In the above proof, we use the following lemma, which is a trivial instance of Grothendieck trace formula and Deligne's weight theory.
\begin{lem}\label{lem-j-dimscheme}
    Fix  positive integers $n,k$ and a geometrically irreducible algebraic scheme $X_0$ over $\mathbb F_q$. Then the following two items are equivalent (note that we denote the pullback of $X_0$ to $\mathrm k$ by $X$): 
    \begin{itemize}
        \item $
    \lim\limits_{i\to \infty}\frac{\#X^{F^{1+ki}}}{q^{(1+ki)n}}=1;
    $
    \item $\dim (X)=n$.
    \end{itemize}
    
\end{lem}

\begin{rmk}
    A counterpart of Theorem \ref{thm-refine-main} in the characteristic $0$ case can be deduced using ingredients introduced in \cite[Section 6]{Kno} and \cite{Kno2}. See Proposition \ref{pro-littleweyl}.
\end{rmk}

\begin{rmk}\label{rm-bxtrivial-general}
    This remark is independent of the rest of this paper. Using the same strategy as in the proof of Theorem \ref{thm-refine-main}, we may show the following:
    \begin{itemize}
        \item Suppose that for a general $x\in\mathcal X=G/H$, the  $B$-stabilizer $B_x$ (with the reduced scheme structure)  is a connected  unipotent group. Then $B_x$ is trivial  for such $x$.
    \end{itemize}
\end{rmk}

\subsection{Conclusion}
We can formulate the main result of this paper.
\begin{thm}\label{thm-main}
    Fix a reductive group $G_0$ (over $\mathbb F_q$), and let $H_0$ be a connected algebraic subgroup of $G_0$. Let $G$ (resp., $H$) be the pullback of $G_0$ (resp., $H_0$) to $\mathrm k$. Suppose further that $G/H$ is a spherical variety with respect to the $G$-action. Then the following statements are equivalent.
    \begin{itemize}
        \item[(1)] The almost multiplicity-one property holds for the pair $(G,H)$ in the sense of Definition \ref{def-almostmulone};
        \item[(2)] The almost multiplicity-one property holds for the pair $(G,H)$ with respect to principal-series representations in the sense of Definition \ref{def-almostmulone-principal};
        \item[(3)] Fix any Borel subgroup $B$ of $G$, we have $i)_B+ii)_B$ as introduced in Remark \ref{rm-changeBcondition};
        \item[(4)] Fix any Borel subgroup $B$ of $G$, we have $i)_B+ii)_B$ as introduced in Remark \ref{rm-changeBcondition}; for the element $w\in B(\mathrm k)\backslash G(\mathrm k)/H(\mathrm k)$ specified in $i)_B$, the algebraic group $B\cap w Hw^{-1}$ endowed with the reduced scheme structure is  trivial.
    \end{itemize}
\end{thm}
\begin{proof}
    (1) implies (2) by Corollary \ref{cor-regularprincipal}. (2) implies (3) by Proposition \ref{pro-principal-efficiency} and Remark \ref{rm-changeBcondition}. (3) implies (1) by Theorem \ref{thm-almostmulone}. (4) implies (3) by the definition. (3) implies (4) by Theorem \ref{thm-refine-main}.
\end{proof}

\section{Counterpart in Characteristic $0$}\label{sec-0-counterpart}
The main reference for this section is \cite{Kno}.
In this section, we consider the characteristic $0$ counterpart of the condition $i)_B+ii)_B$ introduced in Remark \ref{rm-changeBcondition}. The notation adopted in this section is different from that in the rest of this paper.
In this section, $G$ is a connected reductive group over  $\mathbb C$ and $H$ is a connected algebraic subgroup of $G$ such that $\mathcal X:=G/H$ is a $G$-spherical variety. We fix a Borel pair $(B,T)$ of $G$.

For convenience, we list the counterparts of  conditions $i)_B$ and $ii)_B$ from Remark \ref{rm-changeBcondition}:
\begin{itemize}
            \item[$1)_B$] There exists a unique $w\in B(\mathbb C)\backslash G(\mathbb C)/H(\mathbb C)$ such that the reductive quotient of $B\cap w Hw^{-1}$ is a finite algebraic group.
            \item[$2)_B$] Let $w\in B(\mathbb C)\backslash G(\mathbb C)/H(\mathbb C)$ be as in $1)_B$. The group $B\cap w Hw^{-1}$ endowed with the reduced scheme structure is connected and unipotent. 
            \end{itemize}

Note that affine algebraic groups over a field of characteristic $0$ are smooth (see \cite[Theorem 3.23]{Mi}), hence we may rephrase $2)_B$ as: \begin{itemize}
    \item[$2')_B$]  Let $w\in B(\mathbb C)\backslash G(\mathbb C)/H(\mathbb C)$ be as in $1)_B$. The group $B\cap w Hw^{-1}$ is  connected and unipotent. 
\end{itemize}

Since every unipotent algebraic group over a field of characteristic $0$ is connected (see \cite[Proposition 14.32]{Mi}), the term ``connected'' is redundant in $2)_B$ and $2')_B$,  and may be omitted.

\begin{defn}
    As in \cite[Section 6]{Kno}, For any $G$-variety $X$, we define: (see \cite[Section 2]{Kno} for the definition of $c(\cdot)$ and $\mathrm {rk}(\cdot)$.)
\begin{itemize}
    \item $\mathfrak B(X)$ is the set of $B$-stable closed subvarieties of $X$;
    \item $\mathfrak B_{00}(X):=\{Z\in \mathfrak B(X):c(Z)=c(X); \mathrm {rk}(Z)=\mathrm {rk}(X)\}$.
\end{itemize}  
\end{defn}

\begin{rmk}\label{rm-condition-C}
    For a $G$-variety $X$, we formulate the following condition:
    \begin{itemize}
        \item[$(1)_B$] $\mathrm {rk}( X)=\mathrm {rk}(B)$; 
        \item [$(2)_B$] $\mathfrak{B}_{00}( X)$ is a singleton set;
        \item[$(3)_B$]For a general point $x\in  X$, the $B$-stabilizer $B_x$ is a  unipotent group.
    \end{itemize}

\end{rmk}

We may argue as in Corollary \ref{cor-theBorbit} to see: $1)_B+2)_B$ implies that the $B$-orbit $B\cdot wH$ on $G/H$ is open and dense for the element $w$ specified in $2)_B$.
    Unwinding the definitions, we may rephrase $1)_B+2)_B$ as follows:  \begin{itemize}
        \item[] The condition $(1)_B+(2)_B+(3)_B$ holds for $X=\mathcal X$.
    \end{itemize}

Recall the little Weyl group $W_X$ of a $G$-variety $X$, defined in \cite{Kno2}.
Also, recall the parabolic group  $P(X):=\{g\in G:\text{$gBz=Bz$ for general $z\in X$}\}$ of $G$. The following proposition can be viewed as a characteristic $0$ counterpart of Theorem \ref{thm-refine-main}.
\begin{prop}\label{pro-littleweyl}
Fix a smooth $G$-variety $X$.
    Suppose that the condition $(1)_B+(2)_B+(3)_B$ introduced in Remark \ref{rm-condition-C} holds for $X$. Then the little Weyl group $W_X$ is the Weyl group $\mathrm W$ of $G$. Moreover, the stabilizer $B_x$ in $(3)_B$ is the trivial algebraic group. 
\end{prop}

\begin{proof}

Step 1.

We first deal with the case that $X$ is non-degenerate in the sense of \cite[Section 3]{Kno2}. Then $P(X)$ is the largest parabolic subgroup $P$ of $G$ such that all characters in $\chi(X)$ extends to a character of $P$ (see the discussion follows \cite[Theorem 6.2]{Kno}). By the condition $(1)_B$ and $(3)_B$, we see that $\chi(B)=\chi(X)$, forcing $P(X)=B$. In particular, we see that $W_{P(X)}$ the Weyl group of $P(X)$ is the trivial group. By $(2)_B$ and \cite[Theorem 6.2]{Kno}, we see that $W_X=\mathrm {W}$. 

Let $D$ be a $B$-divisor as in  \cite[Section 3, Definition]{Kno2}. By the discussion following \select{loc. cit.}, we see that \cite[Theorem 2.3]{Kno2} applies to the $G$-variety $X$ and the $B$-divisor $D$ (note that $P[D]=P(X)=B$). 
Then \cite[Theorem 2.3]{Kno2} implies that the unipotent radical $U$ of $B$ acts freely on  $X\backslash D$. This shows that the $B$-stabilizer $B_x$ in $(3)_B$ is indeed trivial. And we complete the proof when $X$ is assumed to be non-degenerate.

Step 2.

We deal with a general $X$ satisfying $(1)_B+(2)_B+(3)_B$.
 We choose an effective $B$-divisor $D'$ such that $P[D']=P(X)$ as in \cite[Section 5]{Kno2}. 
We may further assume that $\mathscr O(D')$ is $G$-linearized by replacing $D'$ with a multiple.
Let $p:L\to X$  be the geometric realization of the line bundle $\mathscr O(D')$, with the zero section removed. Then $L$ is a non-degenerate $\mathrm G_m\times G$-variety by \select{loc. cit.}. By the definition of the little Weyl group, we have $W_X=W_L$ (See \cite[Corollary 7.5]{Kno2}). Further, we see that $(1)_{B'}+(2)_{B'}+(3)_{B'}$ holds for $L$ with respect to the reductive group $\mathrm G_m \times G$, where $B'$ is the Borel subgroup $\mathrm G_m\times B$ of $\mathrm G_m\times G$: $(1)_{B'}$ and $(3)_{B'}$ are obvious by the construction of $p$; For $(2)_{B'}$, see the discussion following \cite[Theorem 6.2]{Kno2}. By step 1 (applied to the non-degenerate $\mathrm G_m\times G$-variety $L$), we see that the little Weyl group $W_X$ is exactly the Weyl group of $G$.

It remains to show that for a general $l\in L$, the unipotent group $B_{p(l)}$  is indeed trivial, where $B_{p(l)}$ denotes the $B$-stabilizer of $p(l)$. We fix one such $l$, and let $L_l$ be the fibre of $p$ along $p(l)$. Then the unipotent group $B_{p(l)}$ acts trivially on $L_l$ by the construction of $p$ and $L$. (Note that the fibre $L_l$ is indeed the underlying vector space of a $1$-dimensional algebraic representation of $B_{p(l)}$, with the zero vector removed.) Consequently, the group $B_{p(l)}$ is contained in the $B$-stabilizer $B_l$ of $l$. By step 1, we see that $B_l$ is trivial, forcing $B_{p(l)}$ to be trivial. This completes the proof.
    
\end{proof}


\begin{rmk}
    Parallel to Remark \ref{rm-bxtrivial-general}, we may prove the following statement, using the structure theorem \cite[Theorem 2.3]{Kno2} (note that the structure theorem in \select{loc. cit.} is only applicable to the characteristic $0$ case.):
    \begin{itemize}
        \item Suppose that for a general $x\in G/H$, the $B$-stabilizer $B_x$ is  unipotent. Then $B_x$ is trivial for such $x$.
    \end{itemize}
\end{rmk}

\begin{rmk}\label{rm-strtempered}
    By Proposition \ref{pro-littleweyl} and the preceding discussion, we may rephrase the condition $1)_B+2)_B$ for $\mathcal X$  as (i)+(ii), where we set:
    \begin{itemize}
        \item[(i)]  The little Weyl group $W_\mathcal X$ is the Weyl group $\mathrm W$ of $G$. 
        \item[(ii)] The  $B$-stabilizer $B_x$ for a general $x\in \mathcal X$ is  trivial. 
    \end{itemize}
    Indeed, Proposition \ref{pro-littleweyl} shows that $(1)_B+(2)_B+(3)_B$  implies (i)+(ii). We easily see that (i)+(ii) implies $(1)_B$ and $(3)_B$. Further, by \cite[Theorem 6.2]{Kno2}, the condition (i)+(ii) implies $(2)_B$ for $\mathcal X$. Note that the condition $(1)_B+(2)_B+(3)_B$ is equivalent to $1)_B+2)_B$. (Here, we set the $G$-variety $X$ to be $\mathcal X$ when invoking the conditions listed in Remark \ref{rm-condition-C}.)
\end{rmk}

If the above condition (i)+(ii) holds true for the spherical variety $\mathcal X$, then the conjectural dual group $\check G_\mathcal X$ of $\mathcal X$ in \cite{BZSV} should be the dual group $\check G$ of $G$, indicating that the cotangent bundle $T^* \mathcal X$  is  \textbf{strongly tempered} in the sense of \cite{BZSV} and \cite{WZ}. We remark that in \cite[Section 9]{WZ}, the multiplicity-one property for periods with respect to strongly tempered spherical varieties over local fields has been  studied extensively.

\end{document}